\documentclass[11pt]{amsart}

\usepackage{amsmath,amssymb}
\usepackage{graphicx}
\usepackage{import}
\usepackage{amscd}
\usepackage{wrapfig}
\usepackage{fullpage}
\usepackage{epsfig}
\usepackage{epstopdf}
\numberwithin{equation}{section}

\newtheorem{theorem}{Theorem}[section]

\newtheorem{lemma}[theorem]{Lemma}
\newtheorem{corollary}[theorem]{Corollary}
\newtheorem{proposition}[theorem]{Proposition}
\newtheorem{claim}[theorem]{Claim}

\theoremstyle{definition}
\newtheorem{definition}[theorem]{Definition}

\theoremstyle{remark}
\newtheorem{remark}[theorem]{Remark}

\newcommand{\N}{{\mathbb N}}
\newcommand{\R}{{\mathbb R}}
\newcommand{\Z}{{\mathbb Z}}

\newcommand{\x}{{\mathbf{x}}}
\newcommand{\y}{{\mathbf{y}}}
\newcommand{\m}{\mathbf m}
\newcommand{\w}{\mathbf w}
\newcommand{\dr}{{\mathbf D}}
\newcommand{\J}{{\mathcal J}}
\newcommand{\A}{{\mathcal A}}

\newcommand{\cross}{\times}

\newcommand{\mcg}{\mathcal{MC\kern0.04emG}}
\newcommand{\ml}{\mathcal{M \kern0.07emL}}
\newcommand{\pmf}{\mathcal{P \kern0.07em M \kern0.07em F}}

\newcommand{\teichmuller}{Teichm{\"u}ller{ }}

\makeatletter
 \let\c@theorem=\c@subsection
 \let\c@conjecture=\c@subsection
 \let\c@lemma=\c@subsection
 \let\c@proposition=\c@subsection
 \let\c@claim=\c@subsection
 \let\c@question=\c@subsection
 \let\c@criterion=\c@subsection
 \let\c@vfconj=\c@subsection
 \let\c@definition=\c@subsection
 \let\c@notation=\c@subsection
 \let\c@remark=\c@subsection
 \let\c@example=\c@subsection
 \let\c@equation=\c@subsection
 \let\c@figure=\c@subsection
 \let\c@wrapfigure=\c@subsection

\makeatother

\begin{document}

\title[non-classical interval exchange]{Dynamics of non-classical interval exchanges}
\author[Gadre]{Vaibhav S Gadre}
\address{Mathematics, Harvard University, One Oxford Street, Cambridge,  MA 02138, USA}
\email{vaibhav@math.harvard.edu}

\begin{abstract}
A natural generalization of interval exchange maps are linear involutions, first introduced by Danthony and Nogueira \cite{Dan-Nog}. Recurrent train tracks with a single switch provide a subclass of linear involutions. We call such linear involutions non-classical interval exchanges. They are related to measured foliations on orientable flat surfaces. 

Non-classical interval exchanges can be studied as a dynamical system by considering Rauzy induction in this context. This gives a refinement process on the parameter space similar to Kerckhoff's simplicial systems. We show that the refinement process gives an expansion that has a key dynamical property called {\it uniform distortion}. We use uniform distortion to prove normality of the expansion. Consequently, we prove an analog of Keane's conjecture: almost every non-classical interval exchange is uniquely ergodic. Uniform distortion has been independently shown in \cite{Avi-Res}.
\end{abstract}

\maketitle

\section{Introduction}

Here, we are interested in the dynamical properties of non-classical interval exchanges. In a classical interval exchange, an interval $I$ is partitioned into $d$ subintervals, these subintervals are permuted and glued back preserving orientation to get $I$. This gives a Lebesgue measure preserving map from $I$ to itself. The parameter data that completely determines the map is: first, the lengths of the subintervals and second, the permutation for reshuffling the subintervals. There
is a way to draw these maps pictorially:

\begin{wrapfigure}{r}{0.35\textwidth}
\begin{center}
\ \epsfig{file=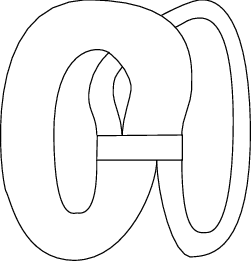, width=0.3\textwidth}
\end{center}
\caption{A classical interval exchange.}
\label{class-iem}
\end{wrapfigure}

We draw the original interval $I$ horizontally and then thicken it vertically to get two copies $I_+$ and $I_-$, the top and the bottom intervals. Divide $I_+$ into $d$ subintervals with the prescribed lengths. Divide $I_-$ also into $d$ subintervals but incorporate the permutation to decide the lengths. Each subinterval of $I_+$  pairs off, by the permutation, with a subinterval of $I_-$ with the same length. Join these subintervals by a band of uniform width equal to their common length. For example, Figure~\ref{class-iem} shows a classical interval exchange with two bands. The interval exchange is exhibited as a map from $I_+$ to $I_+$ given by the vertical flow up the bands from $I_+$ to $I_-$ followed by switching to $I_+$ by identity. The inverse of the interval exchange is then a map from $I_-$ to $I_-$ by flowing down the bands from $I_-$ followed by switching to $I_-$ by identity.

Train tracks enter the discussion as follows: To each band we associate a central edge that goes from $I_+$ to $I_-$. If we collapse each band to it's central edge, then collapse the thickened interval horizontally to a vertical edge and finally collapse the vertical edge, we get a train track with a single switch. The train track has the property that every branch outgoing from one side of the switch is incoming from the other side. A generalization is obtained by relaxing this last property. The associated picture is then a thickened interval with bands in which bands from $I_+$ to $I_+$ or $I_-$ to $I_-$ are allowed. In the terminology of \cite{Dan-Nog}, the resulting dynamical system is a linear involution with no flips. The bands from $I_+$ to $I_+$ or $I_-$ to $I_-$ will be called orientation reversing bands. As before, the space of parameters is the widths of the bands and the exchange combinatorics. 

Here, we restrict to recurrent tracks. This imposes the condition that there are orientation reversing bands on both $I_+$ and $I_-$. We call such linear involutions non-classical interval exchanges. They define a Lebesgue measure preserving dynamical system on $I_+ \sqcup I_-$. See Section~\ref{niem} for precise definitions and details. Non-classical interval exchanges are first return maps  induced on a transverse interval by vertical measured foliations on flat surfaces defined by quadratic differentials that are not squares of abelian differentials. 

In contrast to classical interval exchanges, the widths of bands satisfy an additional constraint: the sum of the widths of the orientation reversing bands on $I_+$ is equal to the sum of the widths of the orientation reversing bands on $I_-$. 

\subsection{Keane's conjecture:}
For classical interval exchanges, Keane conjectured that if the reshuffle permutation is irreducible, then for almost all widths of the bands, the associated exchange is uniquely ergodic with the Lebesgue measure the only invariant probability measure. This conjecture was first proved, independently, by Masur \cite{Mas} and Veech \cite{Vee}. See also Rees \cite{Ree} and Zorich \cite{Zor}. A key tool was Rauzy induction, alternatively called splitting: this is the process of inducing a classical interval exchange on a suitable subinterval by considering the first return map. Veech \cite{Vee} used this to develop a continued fraction expansion for classical interval exchanges. In fact, when $d=2$, the expansion is equivalent to the continued fraction expansion of the ratio of the widths of the two bands. Veech showed that the expansion has most of the nice properties of the classical Gauss map. He derived unique ergodicity as a consequence of this.

A similar conjecture was also made for transverse measured foliations on a Riemann surface. Masur proved this conjecture along with the original Keane conjecture in \cite{Mas}; his approach relied on techniques from \teichmuller theory and applied to both contexts. 

\subsubsection{Kerckhoff's approach:} Later, Kerckhoff \cite{Ker} developed a different and more elementary approach that does not rely on \teichmuller theory, to include both conjectures under a unified setup. In his formulation, Rauzy induction on one hand and train track splittings on the other are both encoded as a refinement process on the associated parameter spaces.

For a classical interval exchange with $d$ bands, the parameter space at each stage can be identified with the standard simplex of dimension $(d-1)$. Rauzy induction splits this simplex into two halves, each a $(d-1)$-simplex itself. A point in the original simplex picks out the half in which it lies. Additionally, there is a map from the new parameter space to the initial parameter space identifying the standard simplex with the half that is picked out. After this identification, the induction is applied to the simplex given by the new stage and the process continues. Associated  to a finite sequence of inductions, there is a map from the current simplex back into the starting simplex. This map is defined iteratively as a composition of maps in the individual steps of the sequence.

It turns out that almost surely, iterations give an infinite expansion consisting of nested simplices. There is a canonical identification of the infinite intersection of these nested simplices with the set of invariant probability measures. The refinement process on a simplex can be defined purely abstractly; Kerckhoff called this a simplicial system and the resulting expansion, a simplicial expansion. He showed that if a simplicial system satisfies a combinatorial condition called the  ``absence of isolated blocks'', then for almost every initial point in the starting simplex, the resulting simplicial expansion is normal: almost surely, every finite sequence of inductions that can occur in the expansion, does occur infinitely often. Finally, normality and the existence of finite sequences which shrink diameter by a definite amount ensure that, almost surely the nested sequence of simplices, actually nests down to a point. To finish the proof for classical interval exchanges, Kerckhoff showed that if the reshuffle permutation is irreducible then the associated simplicial system has no isolated blocks. Consequently, almost surely, the expansion nests down to a point and so there is a unique invariant probability measure. 

For general train tracks, the parameter space is the set of possible weights carried by the train track, normalized so that their sum is one. This is a convex subset of the ambient simplex, cut out by the switch conditions of the train track. In this context, some aspects of the refinement process on the associated parameter space are understood. See \cite{Mos}. However, we do not know how to carry out Kerckhoff's approach in this general setup.

Here, we restrict to non-classical interval exchanges instead. In terms of encoding measured foliations that arise in a given strata of quadratic differentials, there is no loss of generality incurred by this restriction. The advantages are: first, the parameter space has codimension 1 in the standard $(d-1)$-dimensional simplex and is easier to analyze, second, the combinatorics of the Rauzy diagram is better understood by the work of Boissy and Lanneau \cite{Boi-Lan}. The main theorem we prove is:

\begin{theorem}[Normality]\label{normality}
For non-classical interval exchanges, almost surely, the expansion is normal, i.e.~every finite non-transient sequence that can occur in the expansion, does occur infinitely often.
\end{theorem}
\noindent The adjective ``non-transient'' will become clear in Section~\ref{irreducibility}.

As we will see in Section~\ref{dynamics}, combinatorially, there are finitely many subsets inside the standard $(d-1)$-simplex which can be the parameter space for a non-classical interval exchange. We call each such subset a configuration space and denote them by $W_r$. For all subsequent discussion, we fix a combinatorial type $\pi_0$ for the initial non-classical interval exchange and call it the starting stage. The starting stage is assumed to be strongly irreducible and in an attractor of the Rauzy diagram (see Section~\ref{dynamics} for the precise details). Let $W_0$ be the configuration space at $\pi_0$. As we shall see in Section~\ref{dynamics}, for almost every point in $W_0$, iterated Rauzy induction defines an infinite expansion. Let $\pi'$ be a subsequent stage obtained from $\pi_0$ by a finite sequence of inductions. Let $W'$ be the configuration space at $\pi'$. The associated map $f'$ relates the new widths of the bands to the starting widths in $\pi_0$, by giving a diffeomorphism  from $\Delta$ onto a subset of itself, such that $f'(W') = W_0 \cap f'(\Delta)$. Fix a constant $C>1$. The stage $\pi$ is $C$-uniformly distorted if, for any pair of points $\y, \bar{\y}$ in $W$, the Jacobian of $f'$ thought of as a map from $W$ to $W_0$, satisfies
\[
\frac{1}{C} < \frac{\J(f')(\y)}{\J(f')(\bar{\y})} < C
\]
Let $\pi$ be any stage obtained from $\pi_0$ by a finite sequence of inductions. Let $W$ be the configuration space at $\pi$ and $f$ the associated map. The main technical theorem required to prove Theorem~\ref{normality} is the uniform distortion theorem, stated below:

\begin{theorem}[Uniform Distortion]\label{uniform-distortion}
There exists a constant $C>1$, independent of the stage $\pi$, such that, for almost every $\x \in f(W)$, the associated expansion of $\x$ after $\pi$, has some stage $\pi'$, depending on $\x$, such that $\pi'$ is $C$-uniformly distorted. Moreover, the combinatorics of $\pi'$ can be assumed to be the same as $\pi_0$.
\end{theorem}
Theorem~\ref{uniform-distortion} is proved independently in \cite{Avi-Res} as a key step in studying \teichmuller flow on strata of quadratic differentials. See Theorem 4.2 of \cite{Avi-Res}. In \cite{Gad2}, we apply Theorem~\ref{uniform-distortion} to a question about random walks on mapping class groups. In a different direction, Theorem~\ref{uniform-distortion} also fixes a small gap in the proof in \cite{Ker2} that the handlebody limit set has measure zero \cite{Gad3}.

Theorem~\ref{uniform-distortion} implies that the relative probability that a permissible finite sequence occurs right after a uniformly distorted stage, is roughly the same as the probability that an expansion begins with that sequence. Normality is a straightforward consequence of this phenomena.

This approach was outlined by Kerckhoff in \cite{Ker} and carried out for classical interval exchanges. For general train tracks, there are two issues: First, as we shall see in an example in Section~\ref{example}, the probability of a split which is the proportion of the volume of the part of the configuration space that is inside the smaller simplex picked out by the split, can be very different from the ratio of the volumes of the ambient simplices. Second, splitting sequences of general train tracks can have isolated blocks. Because of these issues, the proof of unique ergodicity for measured foliations in \cite{Ker} is incomplete.

We do not know how to fix this for general train tracks. Nevertheless, we resolve these problems for non-classical interval exchanges, by proving Theorems~\ref{uniform-distortion} and~\ref{normality}. As a consequence, we also show in Theorem~\ref{parameter-ergodic} that the map on the parameter space given by Rauzy induction, is ergodic. 

\subsection{Outline of the paper:}
In Section~\ref{niem}, we see how the definition of a classical interval exchange generalizes naturally to a non-classical interval exchange. In Section~\ref{Rauzy}, we define Rauzy induction and explain how it is encoded by matrices. In Section~\ref{Rauzy-diagram}, we define the associated directed graph called the Rauzy diagram. In Section~\ref{irreducibility}, we discuss irreducibility for non-classical interval exchanges. We explain the connection, established in \cite{Boi-Lan}, of the concept of irreducibility to the dynamics of the non-classical interval exchanges and attractors of the Rauzy diagram. We also explain the relevance of these issues to the main theorems, Theorem~\ref{uniform-distortion} and Theorem~\ref{normality}. In Section~\ref{dynamics}, we provide the details for the refinement process on the parameter space. We show that for almost all parameters, we get an infinite expansion. We then formulate the main theorem, Theorem~\ref{Uniform-distortion}, a more precise version of Theorem~\ref{uniform-distortion}. In Section~\ref{example}, we illustrate by an example the main ideas of Kerckhoff's original proof for classical interval exchanges. We also illustrate by the way of an example the key issues that make the task harder for non-classical interval exchanges. One of the issues is that the parameter space for non-classical interval exchanges is codimension 1 in the ambient simplex. Resolving it requires understanding the Jacobian of the restriction of the projective linear map to this subset, instead of considering the full simplex. This is done in Section~\ref{Jacobian}. In Section~\ref{measure-theory}, we recall the main results in the measure theory of projective linear maps with non-negative entries from a standard simplex into itself, and use it to compute the measures of some standard subsets that we encounter later. Section~\ref{proof} gives the detailed proof of Theorem~\ref{Uniform-distortion}. In Section~\ref{proof-of-normality}, we prove normality from Theorem~\ref{Uniform-distortion}. Finally in Section~\ref{unique-ergodicity}, we use normality to show unique ergodicity. As a consequence, we show in Section~\ref{Rauzy-par}, that the Rauzy map on the parameter space is ergodic. 

\subsection{Acknowledgements:}
The research was supported by NSF graduate fellowship under Nathan Dunfield by grant \# 0405491 and \#0707136. This work was done while I was at University of Illinois, Urbana-Champaign. I thank my advisor, Nathan Dunfield, for numerous discussions and his careful perusal of the paper. I thank Steve Kerckhoff and Chris Leininger for helpful conversations during the course of the work. I thank Corentin Boissy and Erwan Lanneau for explaining their results on linear involutions. I am also grateful to the anonymous referees for their detailed comments on earlier drafts.

\section{Non-classical Interval Exchanges}\label{niem}

For a precise definition of a classical interval exchange, see \cite{Yoc}. Here, we focus on representing it pictorially.

Let $\A$ denote an alphabet over $d$ letters. In the definition that follows, the set $\A$ labels the bands. A classical interval exchange is determined by the lengths $(\lambda_\alpha), \alpha \in \A$ of the subintervals and bijections $p_0$ and $p_1$ from $\A$ to the set $\{1, \dotsc, d\}$ as follows: In the plane, draw the interval $I = [0,\sum_\alpha \lambda_\alpha)$ along the horizontal axis and then thicken it slightly in the vertical direction to get two copies, $I_+$ and $I_-$. Call them top interval and bottom interval respectively. Let $\epsilon: I_+ \sqcup I_- \to I_+ \sqcup I_-$ be the map that switches the intervals i.e., $\epsilon(I_+) = I_-$ and $\epsilon(I_-) = I_+$. Subdivide $I_+$ into $d$ subintervals with widths $\lambda_{p_0^{-1}(1)}, \dotsc, \lambda_{p_0^{-1}(d)}$ from left to right. Subdivide $I_-$ into $d$ subintervals with widths $\lambda_{p_1^{-1}(1)}, \dotsc, \lambda_{p_1^{-1}(d)}$ from left to right.  For each $\alpha \in \A$, join the $p_0(\alpha)$ subinterval of $I_+$ to the $p_1(\alpha)$ subinterval of $I_-$ by a band of uniform width $\lambda_\alpha$. The vertical flow along the bands from $I_+$ to $I_-$, followed by $\epsilon$ exhibits the classical interval exchange as a map from $I_+$ to $I_+$. Similarly, the inverse of the interval exchange is realized as a map from $I_-$ to itself by flowing reverse along the bands, followed by $\epsilon$. The ambiguity in the definition at the endpoints of the subintervals is removed by requiring the endpoint flow along the band that lies to the left.

One can construct a train track from this picture. A {\em train track} is a 1-dimensional CW complex with some additional structure. The edges are called {\em branches} and the vertices are called {\em switches}. There is a common point of tangency to all branches meeting at a switch. This splits the set of branches incident at a switch into two disjoint subsets, arbitrarily assigned as incoming and outgoing branches at that switch. Additionally, one assigns non-negative weights to the branches so that the {\em switch conditions} are satisfied: at each switch the sum of the weights of the outgoing branches is equal to the sum of the weights of the incoming branches. 

In our picture, each band has a central edge which joins the midpoint of the subinterval of $I_+$ to the midpoint of the corresponding subinterval of $I_-$. Retract each band to its central edge and the thickened interval to a vertical edge. It is clear that this can be done in a way such that the edges associated to the bands share a vertical line of tangency on each side of the vertical edge. Finally, retract the central vertical edge to a point while preserving the vertical tangency. The result is a train track with a single switch. The branches are in bijection with the bands and every outgoing branch on one side of the switch is incoming from the other. If we assign the width of the band as the weight on the corresponding branch then the weights satisfy the single switch condition.

The first step towards defining non-classical interval exchanges is to relax the constraint that every outgoing branch from one side of the switch is incoming from the other, i.e.~to allow bands from $I_+$ to $I_+$ and $I_-$ to $I_-$. We call such bands {\em orientation reversing} because the flow along such a band reverses the orientation of a subinterval of $I_+$ or $I_-$. 

A train track is {\em recurrent} if there is an assignment of weights satisfying the switch conditions, such that all weights are positive. Here, it implies that if there are orientation reversing bands on one side then there has to be at least one orientation reversing band on the other side. 

\begin{definition}\label{non-c}
A {\em non-classical interval exchange} is the pair of intervals $I_+$ and $I_-$ with bands such that there are orientation reversing bands on both $I_+$ and $I_-$ i.e., the underlying train track is recurrent. The transformation $T: I_+ \sqcup I_- \to I_+ \sqcup I_-$ defined by it is the following composition: Except for the endpoints of the subintervals, every $x \in I_+ \sqcup I_-$ lies in exactly one band. Flow $x$ along this band to its other end to get a point $x'$. Set $T(x)$ to be $\epsilon(x')$. 
\end{definition}

We will no longer distinguish between the picture and the transformation i.e., by the dynamics of a non-classical interval exchange, we will mean the dynamics of $T$. It is clear from the definition, that the Lebesgue measure on $I_+ \sqcup I_-$ is invariant under $T$. 

\subsection{Linear Involutions:}
To relate this to Definition 2.1 in \cite{Boi-Lan}, the map $\tilde{T}$ in their notation is exactly the map given by the flow along the bands. Our requirement that there are orientation reversing bands on $I_+$ and $I_-$ is equivalent in their definition to imposing that there are  subintervals of $I_+$ and $I_-$ that $\tilde{T}$ maps to $I_+$ and $I_-$ respectively. 

The labeling of the bands by $\A$ can be thought of as given by {\em a generalized permutation} as defined in \cite{Boi-Lan}.  A generalized permutation $\pi$ is a 2-1 map from $\{1, \dotsc , 2d\}$ to $\A$. Thus, $\pi^{-1} \alpha$ denotes the two ends of the band labelled $\alpha$. The permutation is of type $(l,m)$ where $l+m=2d$ if the set $\{1, \dotsc, l\}$ enumerates the subintervals of $I_+$ from left to right and the set $(l+1, \dotsc, l+m=2d\}$ enumerates the subintervals of $I_-$ from left to right. . A generalized permutation defines a fixed point free involution $\sigma$ of $\{1, \dotsc, 2d\}$ by:
\[
\pi^{-1}(\pi(i)) = \{ i , \sigma(i)\}
\]
Our definition implies that the generalized permutation we get does not arise from a true permutation $p= p_1p_0^{-1}$ i.e., there is a positive integer $i$ with $i,  \sigma(i) \leqslant l$ and a positive integer $j$ with $l+1 \leqslant j, \sigma(j)$. The equivalence classes under $\sigma$ are indexed by the elements of $\A$ and correspond to the bands. Following Kerckhoff, we shall call the positions that are rightmost on the intervals $I_+$ and $I_-$, the {\em critical} positions.

\begin{wrapfigure}{r}{0.45\textwidth}
\begin{center}
\ \epsfig{file=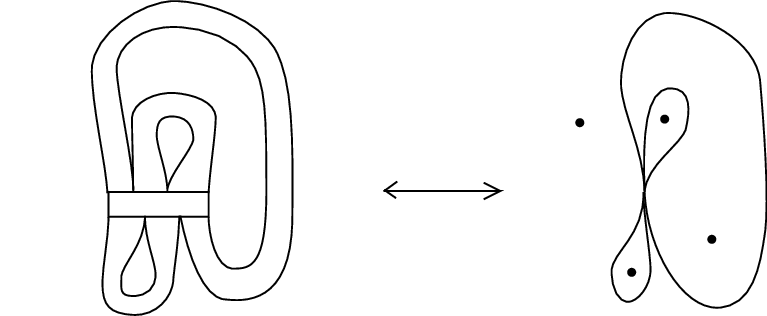, width=0.4\textwidth}
\end{center}
\vskip 5pt \caption{A non-classical interval exchange on $S^2$ minus 4 points.}
\label{iem}
\end{wrapfigure}

\subsection{Non-classical interval exchanges and strata of quadratic differentials:} 
A train track on an oriented surface $S$ with non-negative Euler characteristic, is {\em large} if every region complementary to it in $S$, is a polygon or a once-punctured polygon. Due to the tangency condition at the switches, the complementary regions are {\em ideal} in the sense that the internal angles at all their vertices are zero, and hence they are {\em cusps}. For example, in Figure~\ref{iem}, there are four complementary regions all of which are once-punctured monogons: each contain one puncture and have one cusp. See \cite{Pen-Har} or \cite{Mas-Min} for background on train tracks.

A large train track on $S$ is said to belong to a particular stratum of quadratic differentials if each ideal polygon or once-punctured ideal polygon contains respectively, exactly a single zero or pole, and the number of cusps of each region is the order of the zero or the pole. A non-classical interval exchange belongs to a particular stratum if the underlying train track embeds into $S$ such that the embedding belongs to the stratum. As an example, Figure 17 of \cite{Dun-DTh} shows a non-classical interval exchange in the principal stratum on a 5-punctured sphere, and Figure 19 of \cite{Dun-DTh} shows a non-classical interval exchange in the principal stratum of a genus-2 surface. 

\section{Rauzy induction}\label{Rauzy}

We now describe Rauzy induction. Since the underlying picture of intervals with bands is similar to classical interval exchanges, Rauzy induction for a non-classical interval exchange is similarly defined. The precise definition is given in Section 2.2 of \cite{Boi-Lan}. Here, we concentrate on encoding iterations by products of elementary matrices.

Iterations of Rauzy induction of a classical interval exchange give an expansion analogous to the continued fraction expansion. In fact, when $d=2$, the expansion is exactly the continued fraction expansion of the ratio of the widths of the two bands.

\begin{figure}[htb]
\begin{center}
\ \psfig{file=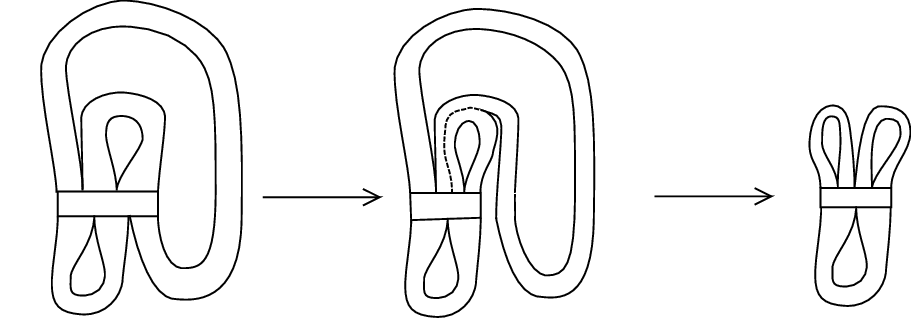, height=1.5truein, width=5truein} \caption{Rauzy Induction} \label{rauzyind}
\end{center}
\setlength{\unitlength}{1in}
\begin{picture}(0,0)(0,0)
\put(-1.85,1.7){$\alpha_0$} \put(-1.3,0.6){$\alpha_1$} \put(1.9,1.7){$\alpha_1$} \put(2.4,1.7){$\alpha_0$} \put(1.9,0.5){$\alpha'_1$} \put(-3,1.2){$\lambda_{\alpha_1} = 3/7$} \put(-3,1){$\lambda_{\alpha_0} = 1/7$} \put(2.5,1.2){$\lambda_{\alpha_1} = 2/7$} \put(2.5,1){$\lambda_{\alpha_0}= 1/7$}
\end{picture}
\end{figure}
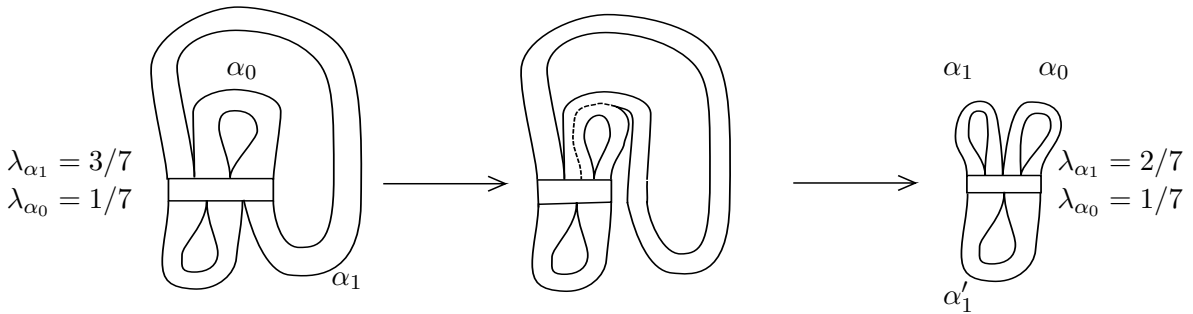

Suppose $T$ is a non-classical interval exchange. Let $\alpha_0$ and $\alpha_1$ be the bands in the critical positions with $\alpha_0$ on $I_+$. First, suppose that $\lambda_{\alpha_0} > \lambda_{\alpha_1}$. Then we slice as shown in Figure~\ref{rauzyind} till we hit $I_+ \sqcup I_-$ for the first time. The $\alpha_0$ band remains in its critical position, but typically a band with a different label $\alpha'_1$ moves into the other critical position. Furthermore, the new width of $\alpha_0$ is $ \lambda_{\alpha_0} - \lambda_{\alpha_1}$. All other widths remain unchanged. If instead $\lambda_{\alpha_1} < \lambda_{\alpha_0}$, then we slice in the opposite direction, which in
Figure~\ref{rauzyind} would be the analogous operation after flipping the picture about the horizontal axis. In either case, we get new non-classical interval exchanges with combinatorics and widths as described above. The operation we just described is called {\em Rauzy induction}.  Since Rauzy induction is represented pictorially by one band being split by another, it's also called a {\it split}. This is consistent with the notion of a split in the context of train-tracks; if the interval exchange is thought of as a train track then Rauzy induction is the same as splitting the train track. Iterations of Rauzy induction are called {\it splitting sequences}.

For a classical interval exchange, Rauzy induction is the first return map to the interval $I' = [0,\sum_{\alpha \neq \alpha_0} \lambda_{\alpha})$ in the first instance and $I' = [0, \sum_{\alpha \neq \alpha_1} \lambda_{\alpha})$ in the second. Similarly, for non-classical interval exchanges, if $I'_+$ and $I'_-$ denote the copies of $I'$ in $I_+$ and $I_-$ respectively, then Rauzy induction is the first return map to $I'_+ \sqcup I'_-$, in either instance.  In this context, not all instances of Rauzy induction are defined. For example:
\begin{enumerate}
\item When $\alpha_0 = \alpha_1$ i.e., both ends of a single band are in the critical positions, neither of the splits are defined. 
\item When $\lambda_{\alpha_0} = \lambda_{\alpha_1}$ then neither of the splits is defined.
\item When $\alpha_0$ is an orientation reversing band on $I_+$ and $\alpha_1$ is the only orientation reversing band on $I_-$, then $\alpha_0$ can split $\alpha_1$ but not the other way round i.e., only one of the splits is defined.
\end{enumerate}
As we shall see later, case (1) is ruled out by the assumption that the non-classical interval exchange is irreducible and case (2) represents a set of measure zero. Eventually, we show in Section~\ref{dynamics}, that almost surely, all iterations of Rauzy induction are well defined. 

 \subsection{Encoding Rauzy expansions by matrices:}

\subsubsection{Description of the parameter space:}
All non-classical interval exchanges sharing the same generalized permutation $\pi$ are parameterized by the widths of the bands. Consider the vector space $\R^{\A}$ and let $\R_{\geq 0}^{\A}$ be the set of points with non-negative coordinates. Let $\Delta$ denote the standard $(d-1)$-simplex in $\R^{\A}$ given by the constrain the sum of the coordinates is 1. An assignment of widths to the bands is a point in $\R^{\A}$. Normalizing the widths so that their sum is 1 restricts us to $\Delta$.  

To be consistent with $\pi$, any assignment of normalized widths must satisfy the switch condition defined by $\pi$. We denote the set of such widths by $W$. Let $\A_+$ and $\A_-$ be the set of orientation reversing bands in $\A$ that are incident on $I_+$ and $I_-$ respectively. Then the points in $W$ satisfy the additional constraint: 
\[
\sum_{\alpha \in \A_+} \lambda_{\alpha}  = \sum_{\alpha \in \A_-} \lambda_{\alpha}
\]
Thus $W$ is the intersection with $\Delta$ of a codimension 1 subspace of $\R^\A$.  For $\alpha \in \A_+$ and $\beta \in \A_-$, let $e_{\alpha \beta}$ be the midpoint of the edge $[e_\alpha, e_\beta]$ of $\Delta$ joining the vertices $e_{\alpha}$ and $e_{\beta}$. The subset $W$ is the convex hull of the points $e_{\alpha \beta}$ and $e_{\rho}$ for $\rho \notin \A_+ \cup \A_-$.

There are finitely many generalized permutations of an alphabet $\A$ over $d$ letters, and hence finitely many convex codimension 1 subsets of $\Delta$ that could be $W$. We call the subsets of $\Delta$ that the generalized permutations define {\em configuration spaces}. Whenever it is necessary to keep track, we index in a fixed manner, the configuration spaces as $W_r$. The full parameter space for non-classical interval exchanges with $d$ bands is a disjoint union of the configuration spaces $W_r$. 

\subsubsection{Matrices:}
Let $I$ denote the $d \times d$ identity matrix on $\R^\A$. For $\alpha, \beta \in \A$, let $M_{\alpha \beta}$ be the $d \times d$-matrix with the $(\alpha,\beta)$ entry 1 and all other entries 0. After Rauzy induction, the relationship between the old and new width data is expressed by
\[
\lambda = E \lambda'
\]
where the matrix $E$ has the form $E = I + M$. In the first instance of the split, when $\lambda_{\alpha_0} > \lambda_{\alpha_1}$, the matrix $M = M_{\alpha_0\alpha_1}$; in the second instance of the split, when $\lambda_{\alpha_1} > \lambda_{\alpha_0}$, the matrix $M = M_{\alpha_1 \alpha_0}$. Thus, in either case the matrix $E$ is an elementary matrix, in particular $E \in SL(d; \Z)$. If $B$ is any $d \cross d$ matrix then in the instance when $\lambda_{\alpha_0} > \lambda_{\alpha_1}$, the action on $B$ by right multiplication by $E$ has the effect that the
$\alpha_1$-th column of $B$ is replaced by the sum of the $\alpha_0$-th column and $\alpha_1$-th column of $B$. We phrase this as: in the split, the $\alpha_0$-th column {\em moves} $\alpha_1$-th column. Similar statement holds when $\lambda_{\alpha_0} > \lambda_{\alpha_1}$.

\section{Rauzy diagram}\label{Rauzy-diagram}

For non-classical interval exchanges, one constructs an oriented graph similar to the Rauzy diagram for a classical interval exchange. However, as we shall see there are some key differences in the non-classical context.

Construct an oriented graph $G$ as follows: the nodes of the graph are generalized permutations $\pi$ of an alphabet $\A$ over $d$ letters satisfying the conditions imposed by Definition~\ref{non-c} and summarized in Section 2.2. We draw an arrow from $\pi$ to $\pi'$, if $\pi'$ results from splitting $\pi$. For each node $\pi$, there are at most two arrows coming out of it. A splitting sequence gives us a directed path in $G$.

In the context of classical interval exchanges, irreducibility of the permutation is sufficient to ensure that each connected component of the Rauzy diagram is an {\em attractor} i.e., any node can be joined to any other node by a directed path. Each component is called a Rauzy class. Veech showed a bijective correspondence between extended Rauzy classes and connected components of the corresponding strata of abelian differentials \cite{Vee}, and Kontsevich and Zorich gave a classification scheme for them \cite{Kon-Zor}. 

The Rauzy diagram for non-classical interval exchanges is more complicated and need not have such strong recurrence properties. See the examples in the Appendix of \cite{Boi-Lan} or see Section 10 of \cite{Dun-DTh}. The notion of strong irreducibility of generalized permutations defined by Boissy-Lanneau \cite{Boi-Lan} is needed to characterize attractors. We will explain these issues and their relevance to Theorem~\ref{Uniform-distortion} in the next section.

\section{Irreducibility}\label{irreducibility}

\noindent Let $\vert A \vert$ denote the Lebesgue measure of a measurable subset $A \subset I_+ \sqcup I_-$. We shall first consider the simplest notion of reducibility:

\begin{definition}\label{comb-reducible}
A generalized permutation $\pi$ is {\em combinatorially reducible} if $\A$ can be written as a disjoint union $\A_1 \sqcup \A_2$ of nonempty subsets such that for all $\alpha \in \A_1$ and $\beta \in \A_2$, the ends $\pi^{-1}(\alpha)$ of band $\alpha$ occur to the left of the ends $\pi^{-1}(\beta)$ of the band $\beta$.
\end{definition}

This means that for {\em some} choice of widths, the intervals $I_\pm$ can be cut into two intervals $I_\pm (1) \sqcup I_\pm (2)$ such that $\vert I_+ (1) \vert = \vert I_- (1) \vert$ and $\vert I_+(2) \vert = \vert I_-(2) \vert$ and all the $\A_1$ bands are incident on $I_\pm(1)$ and all the $\A_2$ bands are incident on $I_\pm(2)$. Thus, $I_\pm(1)$ and $I_\pm(2)$ define interval exchanges by themselves with $\# \A_1$ and $\# \A_2$ bands respectively i.e., for some choice of widths, the original exchange is obtained by concatenating two disjoint exchanges with fewer bands. In particular, the dynamics of the original exchange is non-minimal.

\begin{definition}\label{measure-reducible}
A generalized permutation $\pi$ is said to be {\em measure reduced} if for all widths of the bands, the non-classical interval exchange decomposes into disjoint exchanges given by the $\A_1$ and $\A_2$ bands respectively. 
\end{definition}

A generalized permutation $\pi$ is {\em irreducible} if it has no measure reduction. 

\begin{lemma}
A generalized permutation $\pi$, combinatorially reducible as $\A= \A_1 \sqcup \A_2$, is measure reduced if and only if all the orientation reversing bands in $\pi$ are entirely in $\A_1$ or entirely in $\A_2$.
\end{lemma}

\begin{proof}
Let $\A_{1,-}$ and $\A_{1,+}$ denote the set of orientation reversing bands in $\A_1$ that are incident on the bottom and top respectively. The generalized permutation $\pi$ is measure reduced if and only if the additional constraint
\begin{equation}\label{red-constraint}
\sum_{\alpha \in \A_{1,-}} \lambda_\alpha = \sum_{\beta \in \A_{1,+}} \lambda_\beta
\end{equation}
gives the defining equation of $W$ or is vacuous. If the former, then all the orientation reversing bands of $\pi$ have to be in $\A_1$. If the latter, then all the orientation reversing bands of $\pi$ have to be in $\A_2$.
\end{proof}

A train track on a surface $S$ with non-negative Euler characteristic is {\em complete} if all its complementary regions are ideal triangles or once punctured monogons. This means that a generic measured foliation carried by it is the vertical foliation of a quadratic differential in the principal stratum i.e., have simple zeroes and poles. We shall call a non-classical interval exchange {\em complete} if the underlying train track embeds as a complete train track on some surface $S$. Technically, the train track is required to be transversely recurrent in addition to being recurrent but transverse recurrence of non-classical interval exchanges follows by a direct application of Corollary 1.3.5 in \cite{Pen-Har}. So, we skip the definition of transverse recurrence and refer the reader to \cite{Pen-Har}. In the proposition below, we show that in the principal stratum, combinatorial reducibility is never measure reducibility. 

\begin{proposition}\label{comb-meas}
A generalized permutation $\pi$ of a complete non-classical interval exchange is irreducible. 
\end{proposition}
\begin{proof}
Suppose that $\pi$ admits a measure reduction $\A = \A_1 \sqcup \A_2$. By the previous lemma, all the orientation reversing bands in $\pi$ must belong entirely say to $\A_2$. Mark the points $p_\pm$ on $I_\pm$ that are the common endpoints of the intervals $I_\pm(1)$ and $I_\pm(2)$. Suppose $\alpha$ and $\beta$ are the bands incident on $I_+(1)$ and $I_+(2)$ that are adjacent at the common endpoint $p_+$. Let $\mathcal{R}$ be the complementary region with a cusp at $p_+$ and $\alpha$ and $\beta$ as sides. For $\mathcal{R}$ to be a once-punctured monogon, $\alpha$ has to be the same as $\beta$ which is not possible since $\pi$ is measure reduced. So, $\mathcal{R}$ has to be an ideal triangle. Then, the remaining two cusps of $\mathcal{R}$ have to be incident on $I_-(1)$ and $I_\pm(2)$ respectively. But this implies that there is a band whose one end is incident on $I_-(1)$ and its other end is incident on $I_\pm(2)$. This contradicts the fact that $\pi$ is measure reduced. 
\end{proof}

\subsection{Strong irreducibility and attractors:} \label{StrongIrr}
In contrast to Proposition~\ref{comb-meas}, in some of the other strata, there are combinatorially reducible non-classical interval exchanges that are also measure reduced. This can give rise to the following phenomena: The Rauzy diagram may contain generalized permutations that are not combinatorially reducible but which split with positive probability to measure reduced ones. In fact in some cases, the dynamics of the exchange is minimal for a set of widths of intermediate measure, and non-minimal because of a measure reduction for the complementary widths. See Figure 15 in \cite{Boi-Lan} for an example. This makes the attractors of the Rauzy diagram are harder to characterize.
 
In \cite{Boi-Lan}, Boissy and Lanneau define a stronger notion of reducibility and prove:
\begin{theorem}[\cite{Boi-Lan} Theorem C ]\label{attractors}
Let $G_{irr}$ be the subset of nodes of the Rauzy diagram $G$ corresponding to the strongly irreducible generalized permutations. Then $G_{irr}$ is closed under forward iterations of Rauzy induction. Moreover, each connected component of $G_{irr}$ is strongly connected i.e., any node $\pi$ in a connected component of $G_{irr}$ can be connected to any other node $\pi'$ in the same component of $G_{irr}$ by a sequence of splits.
\end{theorem}

The theorem above implies that a strongly irreducible non-classical interval exchange can never split to a measure reduced one. Thus, by restricting to strongly irreducible non-classical interval exchanges, one can avoid the issues mentioned.

From the point of view of Uniform Distortion i.e., Theorem~\ref{Uniform-distortion}, these issues become relevant to Propositions~\ref{inv-measures} and~\ref{critical}. For Proposition~\ref{inv-measures}, minimality of the non-classical interval exchange is necessary. For Proposition~\ref{critical}, it is necessary to know that the set of widths for which a combinatorial reduction decomposes a non-classical interval exchange into two exchanges with fewer bands, has measure zero i.e, it is necessary to know that the expansion never gives a measure reduced generalized permutation. As shown in \cite{Boi-Lan}, strong irreducibility implies both these properties. 

For us, it is enough to assume that the initial generalized permutation $\pi_0$ satisfies the following possibly weaker properties: first, the set of widths that define a minimal non-classical interval exchange has full measure, and second, the generalized permutation $\pi_0$ never splits to a measure reduced one. We do not use the actual definition of strong irreducibility at any point in the argument. Hence, we shall skip the details of the definition, and refer the reader to \cite{Boi-Lan}. It should be pointed out that in the principal stratum, all generalized permutations split in finite time to strongly irreducible ones. So,  in the principal stratum strong irreducibility is not required to prove Theorem~\ref{Uniform-distortion}. 

Similarly for Theorem~\ref{normality}, it is enough to assume that the initial generalized permutation $\pi_0$ is such that all splitting sequences starting from $\pi_0$ end up in some attractor. 

The assumption $\pi_0$ is strongly irreducible implies all the required hypothesis. So, for the rest of the paper, we will assume that the initial generalized permutation $\pi_0$ is strongly irreducible, in which case there are no transient nodes at all.

\section{Dynamics}\label{dynamics}

In this section, we analyze the expansion by splitting sequences on the space of widths of the bands. It turns out that the space of invariant probability measures embeds into the space of widths and the iterative refinement process defined by the splitting sequences, in the limit, nests down to the set of invariant probability measures. So the basic idea, due to Kerckhoff, is to show that for almost every non-classical interval exchange, one nests down to a single point.

\subsubsection{Preliminary notation:}
Given a matrix $A$ with non-negative entries, we define the projectivization $\Gamma A$ as a map from $\Delta$ to itself by
\[
\Gamma A(\y) = \frac{A \y}{\vert A \y \vert}
\]
where if $\y = (y_1, y_2, \cdots, y_d)$ in coordinates then $\vert \y \vert = \sum \vert y_i \vert$. This shall be the norm used throughout. The norm is additive on $\R_{\geq 0}^{d}$, i.e.~for $\y, \y'$ in $\R_{\geq 0}^{d}$, $\vert \y + \y'\vert = \vert \y \vert + \vert \y' \vert$.

\subsection{Iterations of Rauzy induction:}
We fix a generalized permutation $\pi_0$ that is strongly irreducible, and hence belongs to an attractor of $G$. Non-classical interval exchanges with generalized permutation $\pi_0$ are points in the configuration space $W_0$ defined by $\pi_0$. Let $\x = (\lambda_\alpha)$ be a point in $W_0$.  First, we show that almost every $\x \in W_0$ has an infinite expansion.

Recall from Section 3 that Rauzy induction stops in Cases (1) and (2). if Case (1) is true then the underlying generalized permutation is measure reduced, which rules out Case (1). In Case (2), except when the critical bands $\alpha_0$ and $\alpha_1$ are the only orientation reversing bands on $I_+$ and $I_-$, the set of widths satisfying $\lambda_{\alpha_0} = \lambda_{\alpha_1}$ form a codimension 1 subset of the associated configuration space. If $\alpha_0$ and $\alpha_1$ are the only orientation reversing bands on $I_+$ and $I_-$ then we can simply amalgamate together the ends of $\alpha_0$ and $\alpha_1$ in the critical positions into a single band and cut those parts out of the intervals $I_\pm$. This makes it equivalent to a classical interval exchange with $(d-1)$ bands for all widths contradicting the assumption that the generic vertical foliation carried, is non-orientable. The conclusion is that when the expansion stops in finite time, the widths belong to a codimension 1 subset of the associated configuration space. 

Finally the subset of points whose Rauzy expansion stops in a finite number of steps is a countable of union of codimension 1 sets, and hence measure zero. So, for almost every $\x$ iterated splitting gives an infinite expansion. An infinite expansion determines an infinite directed path in the attractor. 

A finite directed path $\pi_0 \to \pi_1 \to \dotsc \to \pi_n$ in the attractor shall be called a {\em stage}. Let $E_i$ denote the elementary matrix associated to the split $\pi_{i-1} \to \pi_{i}$. Let $W_i$ denote the configuration spaces corresponding to the $\pi_i$. The projective linear maps $\Gamma E_i$ have the property that the inverse image of the configuration space $W_{i-1}$ is the configuration space $W_i$. The  matrix $Q_n$ associated to the stage is given by the product
\[
Q_n = E_1 E_2 \dotsc E_n
\]
The image of $\Gamma Q_n$ in $W_0$ i.e., the set $\Gamma Q_n(W_n)$ is the set of all $\x \in W_0$ whose expansion begins with this finite sequence $\pi_0 \to \pi_1 \to \dotsc \to \pi_n$.

For a point $\x$ with an infinite expansion, whenever it is necessary to emphasize the dependence of the directed path on the initial point $\x$, we shall denote the nodes in the directed path by $\pi_{\x,i}$, the configuration spaces defined by $\pi_{\x, i}$ by $W_{\x,i}$, and the elementary matrices associated to the splits $\pi_{\x,i-1} \to \pi_{\x,i}$ by $E_{\x,i}$. Thus, given a stage $\pi_0 \to \pi_1 \to \dotsc \to \pi_n$, the set $\Gamma Q_n(W_n)$ is precisely the set of all $\x \in W_0$ for which $\pi_{\x,i} = \pi_i$ for all $i \leq n$. 

In the expansion for $\x$, the actual (or un-normalized) widths $\lambda^{(n)}$ at any stage are related by the equation
\[
\x = Q_n \lambda^{(n)}
\]
The projectivization $\x^{(n)} = \lambda^{(n)}/ \vert \lambda^{(n)} \vert$ lies in the configuration space $W_n$. Thus, we get a sequence of points $\x^{(n)} \in W_n$ such that $\x = \Gamma Q_{\x,n}  \x^{(n)}$. The sets $\Gamma Q_{\x,n}(W_n)$ form a nested sequence in $W_0$, all containing $\x$. Let 
\[
C(\x) = \bigcap_{n} \Gamma Q_{\x,n}(W_n)
\]
Let $\mu$ be a probability measure on the disjoint union $I_+ \sqcup I_-$ invariant under the non-classical interval exchange $T$ defined by $\x = (\lambda_\alpha)$. Let $(\lambda^\mu_\alpha)$ be the widths assigned by $\mu$ to the bands. If $T$ is minimal, then $\mu$ is absolutely continuous with respect to $l$, where $l$ is the standard Lebesgue measure on $I_+ \sqcup I_-$. Let $t \in I_+ \sqcup I_-$ and consider the subinterval $[0,t)$ in the same component $I_+$ or $I_-$ as $t$. Set $H_\mu(t) = \mu([0,t))$ to be in the same component as $t$. The map $H_\mu$ is a homeomorphism of $I_+ \sqcup I_-$. Define $T_\mu = H_\mu \circ T \circ H^{-1}_\mu$. It is easy to see that the transformation $T_\mu$ is a non-classical interval exchange with generalized permutation $\pi_0$. Following the exact argument as the first proposition of Section 4.4 of \cite{Yoc}, we get

\begin{proposition}\label{inv-measures}
The map $\mu \to (\lambda^\mu_\alpha)$ is a linear homeomorphism from the set of $T$-invariant probability measures onto the set $C(\x)$. In particular, the non-classical interval exchange $T$ is uniquely ergodic if and only if $C(\x) = \x$. 
\end{proposition}

Subsequently, we will be interested in estimating the Lebesgue measure of subsets of $\Gamma Q_n(W_n)$. Here, Lebesgue measure means the probability measure on $W_0$ given by the $(d-2)$-volume form induced on it as a sub-manifold of $\Delta$, normalized so that the total volume of $W_0$ is 1. We shall denote it by $\ell$. 

For example, to get an estimate of $\ell(\Gamma Q_n(W_n))$, we first push-forward, by $\Gamma Q_n$, the volume form on $W_n$. Since there are finitely many configuration spaces, the volumes of any two configuration spaces $W_r$ and $W_s$ differ up to some factor that depends only on $d$. So now compare the actual measure on $W_0$ to the push-forward. The Radon-Nikodym derivative of the actual measure with respect to the push-forward is just the Jacobian of $\Gamma Q_n$, restricted as a map from $W_n$ to $W_0$. So integrating the Jacobian over $W_n$ gives us
$\ell(\Gamma Q_n(W_n))$ up to the factor that relates the volumes of the two configuration spaces. This shows that to give quantitative estimates, one needs to understand the Jacobian of $\Gamma Q_n$ restricted as a map from $W_n$ to $W_0$. We denote this Jacobian by $\J(\Gamma Q_n)$.

Suppose $\pi_n$ is the same as $\pi_0$ at some stage in the expansion and suppose $\jmath$ is a finite splitting sequence starting from $\pi_0$. If the Jacobian $\J(\Gamma Q_n)$ is roughly the same at all points, then the relative probability that $\jmath$ follows $\pi_n$ is also roughly the same as the probability that an expansion starts with $\jmath$. We make the notion of the Jacobian being roughly the same at all points, precise below.

\begin{definition}\label{C-distortion}
Suppose $\pi_0 \to \pi_1 \to \dotsc \to \pi_n$ is a finite directed path in the attractor and $Q_n$ the associated matrix. For $C > 1$, we say that the stage $\pi_n$ is {\em $C$-uniformly distorted} if for all $\y, \y' \in W_n$
\[
\frac{1}{C} \leq \frac{\J(\Gamma Q_n)(\y)}{\J(\Gamma Q_n)(\y')} \leq C
\]
\end{definition}

At this point, we adopt some conventions: At any stage, points in $W_0$ or more generally in its ambient simplex shall be denoted by $\x$'s, points in $W_n$ or more generally in its ambient simplex by $\y$'s, points in $Q_n W_n$ or more generally in the image under $Q_n$ of the ambient simplex by $\w$'s. In all cases, we use suitable subscripts whenever necessary. From the previous discussion $\y = \x^{(n)}$, where $\x^{(n)}$ is the $n$-th point constructed iteratively in the expansion of $\x$. The main technical theorem is

\begin{theorem}\label{Uniform-distortion}
Suppose $\pi_0 \to \pi_1 \to \dotsc \to \pi_n$ is a stage in the expansion. There exists a constant $C>1$, independent of the stage, such that for almost every $\x \in \Gamma Q_n(W_n)$, there is some $m>n$, depending on $\x$, such that the stage $\pi_{\x,m}$ is $C$-uniformly distorted.
\end{theorem}

\section{Kerckhoff's approach: Some examples}\label{example}

We fix the following notation: Let $Q$ be the matrix associated to a stage in the expansion. We denote the $\alpha$-th column of $Q$ by $Q(\alpha)$.

Before we present examples, we note that for classical interval exchanges Equation~\eqref{full-jaco} implies that $C$-uniform distortion is equivalent to the $C^{1/d}$-distribution of the columns of the associated matrix i.e., the ratio of the norms of any two columns of $Q$ must be in the interval $(1/C^{1/d}, C^{1/d})$.

As the first example, consider an irreducible classical interval exchange with two bands. There is just one possible combinatorial type and hence the Rauzy diagram has just a single vertex. At the starting stage, we shall label the band in the critical position on the bottom as 1 and the other band as 2. At every stage, we normalize so that the sum of the widths of the bands is one. If at any stage, band 1 splits band 2, then we denote the split by the symbol $L$ and if it happens the other way round, then we denote it by $R$. The matrices corresponding to $L$ and $R$ which we also denote by the same letters are:
\[
L = \left[ \begin{array}{cc} 1 & 1\\0 & 1 \end{array} \right], \hspace{10mm} R = \left[ \begin{array}{cc} 1& 0 \\ 1 & 1\end{array} \right]
\]
Suppose $Q$ is the matrix at some stage in the expansion and suppose that $\vert Q(2) \vert > \vert Q(1) \vert$. Subsequent to this stage, as long as the split $L$ keeps occurring, the column $Q(1)$  keeps moving column $Q(2)$. Hence the new matrix $Q'$ has columns
\[
Q'(2) = Q(2) + k Q(1), \hspace{5mm} Q'(1) = Q(1)
\]
where $k$ is the number of times $L$ has occurred. In this case, the norm of the second column keeps increasing while the first column stays the same, making the inequality $\vert Q'(2) \vert > \vert Q'(1) \vert$ more and more pronounced. However, as soon as $R$ occurs, the second column moves the first, and by additivity of the norm, we get 
\[
\frac{1}{2} < \frac{\vert Q'(1) \vert}{\vert Q'(2) \vert} < 2
\]
i.e., the columns become 2-distributed. To summarize, as long as a sequence of $L$'s occur, the columns get farther and farther from being nicely distributed and the resulting stage is farther and farther from being uniformly distorted. But as soon as a $R$ occurs after that, the columns get 2-distributed, resulting in a stage that is uniformly distorted.

Now using the measure theory of projective linear maps from a 1-simplex to itself, specifically Equation~\eqref{wedge-measure}, it can be shown that the probability of the second column increasing in norm by a factor of $K$ due to a sequence of $L$'s, is bounded away from 1 by a quantity that depends only on $K$ and is independent of the stage. This implies that with a definite probability, the split $R$ must occur, giving us 2-distribution and an instance of the theorem.

To handle classical interval exchanges with $d$ bands, Kerckhoff \cite{Ker} first proves a similar proposition about {\it increase in norm}: At any stage, the probability that a band is never split, before the norm of its column increases by a large enough factor $K$, is bounded away from 1. The bound depends only on $K$ and $d$ and is independent of the stage. In fact, the bound monotonically goes to zero as $K$ goes to infinity. As in the example above, the proof of this proposition uses Equation~\eqref{wedge-measure}.

After this, Kerckhoff proves the following inductive step: Suppose our stage has a collection of $C'$-distributed columns that also includes the column with the largest norm. Then there is a definite probability that one gets a larger collection of $C''$-distributed columns that also includes the largest column at the new stage. The constant $C''$ depends only on $C'$ and $d$ and is independent of the stage. Iterating the inductive step shows that with a definite probability, one must get $C$-distributed.

The basic idea behind the inductive step is: Before the norms of the columns in the collection increase by a factor of $K$, suppose one of the following happens: either an outside column becomes the column with the largest norm or a column in the collection moves a column from outside. At this point, if we add the outside band to our collection, then similar to the example above, the ratio of the norms of any two columns in it is within suitable bounds, even if initially, the ratio is way out. So it remains to show that with a definite probability, one of the two
events happens.

The key idea is that this happens provided there are no isolated blocks. An {\em isolated block} is a splitting sequence in which there is a collection of bands that satisfy the properties: First, every band in the collection is moved, at least once, by some other band in the collection and second, every band in the collection moves some other band in the collection but never moves a band outside the collection. Kerckhoff shows that for irreducible classical interval exchanges, isolated blocks are absent.

The proposition about increase in norm implies that with a definite probability, every band in the collection has to move some other band before its norm increases by $K$. But it could so happen that it moves a band in the collection itself. If this repeats enough number of times, then Kerckhoff shows that there is a sub-collection that forms an isolated block. The number of times it needs to repeat is independent of the stage. Finally, since isolated blocks are absent, there is a definite probability that a band in the collection has to move a band outside, which is exactly the
kind of split we want to finish the proof of the inductive step.

With non-classical interval exchanges, there are two issues: First, the probabilities of splits are different from the ratios of the measures of the ambient simplices. Because of this, for non-classical interval exchanges, the proposition about increase in norm is incorrect as it stands. Second, isolated blocks are possible. 

To illustrate the first issue, consider the stage given by the splitting sequence in Figure~\ref{seqn}. Let $\Delta$ denote the standard simplex in $\R^4$, and let $Q$ denote the matrix associated to the stage. Now suppose that following this stage, band 1 splits band 2. Let $E$ denote the elementary matrix associated to this split. In the subsequent computation, we show that the probability that this split happens is $(n+3)/2(n+2)$, which approaches $1/2$ as $n$ becomes large. On the other hand, when band 1 splits band 2, the column $Q(2)$ moves the column $Q(1)$. The computation shows that this increases $\vert Q(1) \vert$ by a factor of $2(n+2)/3$, which is unbounded as $n$ becomes large. It is also interesting to compare the probability of the split to the ratio $\ell(T(QE)(\Delta))/\ell(\Gamma Q(\Delta))$ of the volumes of the ambient simplices. By Lemma~\ref{measure-plm}, the ratio $\ell(T(QE)(\Delta))/\ell(\Gamma Q(\Delta))$ is equal to the reciprocal of the factor by which $\vert Q(1) \vert$ increases, i.e. equal to $3/2(n+2)$. This is completely different from the probability of the split. The example thus illustrates that for non-classical interval exchanges Proposition 1.1 from Kerckhoff \cite{Ker} cannot hold as it stands. 

However, a closer look at the proof of the inductive step reveals that it suffices that the proposition about increase in norm be true for the {\em largest} columns around. In Proposition~\ref{measure-norm}, we show that this is indeed the case for non-classical interval exchanges.

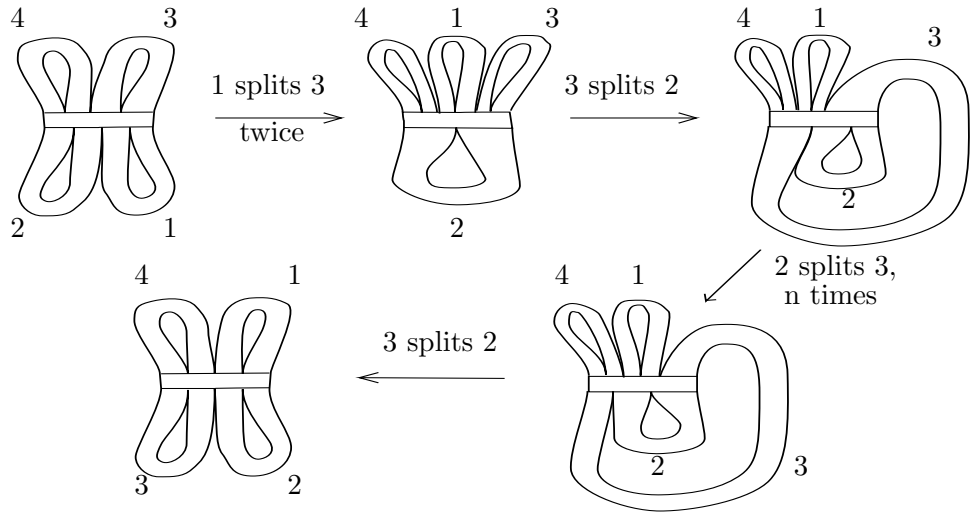
\begin{figure}[htb]
\begin{center}
\ \psfig{file=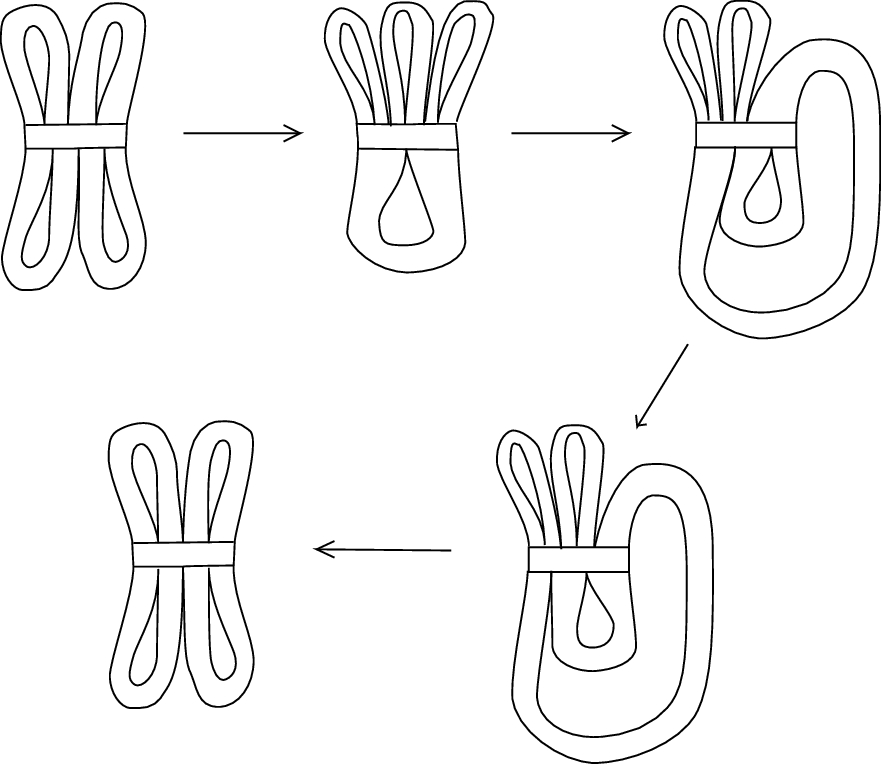, width=5truein, height=2.5truein} \caption{Splitting Sequence} \label{seqn}
\end{center}
\setlength{\unitlength}{1in}
\begin{picture}(0,0)(0,0)
\put(-1.45,2.75){1 splits 3}\put(-1.3,2.5){twice} \put(0.4,2.75){3 splits 2} \put(1.5,1.8){2 splits 3,} \put(1.55,1.65){n times} \put(-0.55,1.4){3 splits 2} \put(-2.5,3.1){4}\put(-1.7,3.1){3}\put(-0.7,3.1){4}\put(-0.2,3.1){1}\put(0.3,3.1){3}\put(1.3,3.1){4}
\put(1.7,3.1){1}\put(-2.5,2){2}\put(-1.7,2){1}
\put(-0.2,2){2}\put(1.85,2.15){2}\put(2.3,3){3}\put(-1.85,1.75){4}\put(-1.05,1.75){1}\put(0.35,1.75){4}
\put(0.75,1.75){1}\put(-1.85,0.65){3} \put(-1.05,0.65){2}\put(0.85,0.75){2}\put(1.6,0.75){3}
\end{picture}
\end{figure}

\noindent The computations in the example now follow. The matrix $Q$ is given by
\[
Q = \left[ \begin{array}{cccc} 1 & 0 & 0 & 0\\ 0 & n+1 & n+2 & 0\\2 & n & n+1 & 0\\0 & 0 & 0 & 1 \end{array} \right]
\]
Denote the simplex with the columns of $Q$ as vertices by $\Delta(Q)$. Let the image under $Q$, of the configuration space for the stage be $W(Q)$.

As in section~\ref{dynamics}, let $e_{ij}$ be the midpoint of the edge $[e_i,e_j]$. As shown in Figure~\ref{initialconfig}, in the standard simplex, the initial configuration space $W$ is the quadrilateral with vertices $e_{13}, e_{14}, e_{24}$ and $e_{23}$.

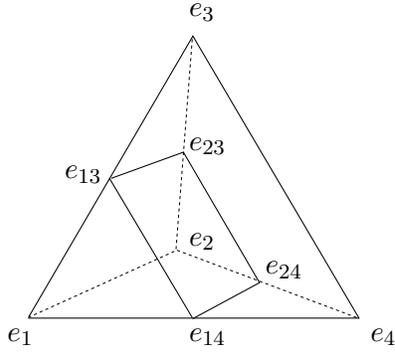
\begin{figure}[htb]
\begin{center}
\ \psfig{file=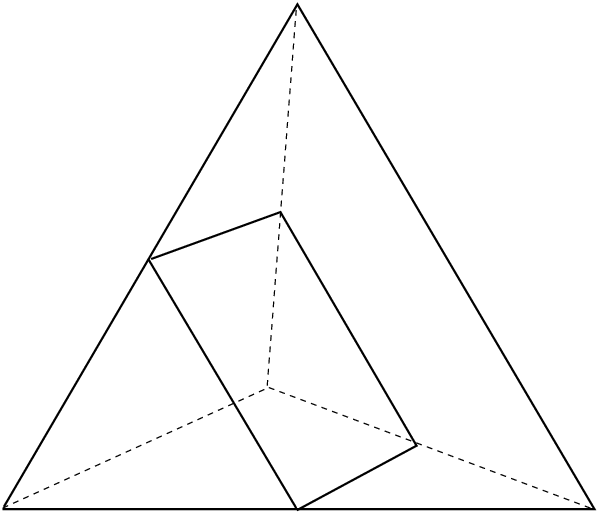, height=1.5truein} \caption{$W$ in $\Delta$} \label{initialconfig}
\end{center}
\setlength{\unitlength}{1in}
\begin{picture}(0,0)(0,0)
\put(-0.95,0.45){$e_1$} \put(0,0.45){$e_{14}$} \put(0.95,0.45){$e_4$} \put(0,2.15){$e_3$} \put(0,0.95){$e_2$} \put(0,1.45){$e_{23}$}
\put(0.4,0.8){$e_{24}$} \put(-0.65,1.3){$e_{13}$}
\end{picture}
\end{figure}

Let $Q(ij)$ denote the midpoint of the edge $[Q(i),Q(j)]$ of $\Delta(Q)$. As shown in Figure~\ref{stage}, $W(Q)$ is the quadrilateral in $\Delta(Q)$ with vertices $Q(12), Q(13), Q(34)$ and $Q(24)$.

\begin{figure}[htb]
\begin{center}
\ \psfig{file=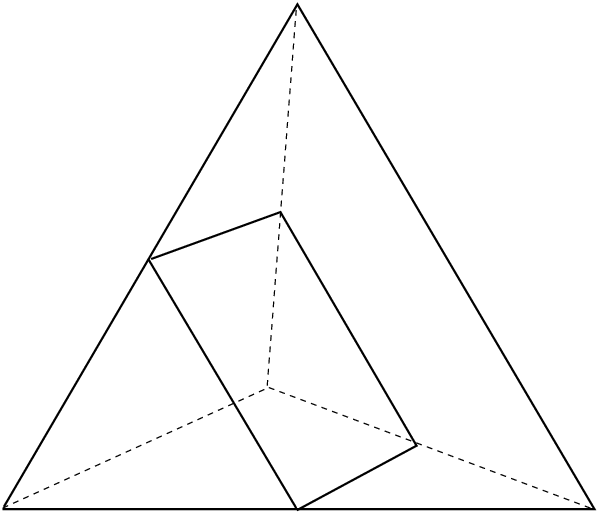, height=1.7truein} \caption{$W(Q)$ in $\Delta(Q)$} \label{stage}
\end{center}
\setlength{\unitlength}{1in}
\begin{picture}(0,0)(0,0)
\put(-1.2,0.4){$Q(1)$} \put(0,0.4){$Q(13)$} \put(1,0.4){$Q(3)$} \put(0,2.35){$Q(2)$} \put(-0.15,1.65){$Q(24)$} \put(0.45,0.75){$Q(34)$}
\put(-0.95,1.4){$Q(12)$} \put(-0.15,1){$Q(4)$}
\end{picture}
\end{figure}

The columns representing the vertices of $W(Q)$ are
\begin{eqnarray*}
Q(12) &=& \left[ \begin{array}{c} 1/2 \\ (n+1)/2 \\(n+2)/2\\ 0 \end{array}\right], Q(13) = \left[ \begin{array}{c} 1/2 \\ (n+2)/2 \\(n+3)/2\\ 0
\end{array}\right] \\
Q(34) &=& \left[ \begin{array}{c} 0 \\ (n+2)/2 \\(n+1)/2\\ 1/2 \end{array}\right], Q(24) = \left[ \begin{array}{c} 0\\(n+1)/2\\n/2\\1/2
\end{array} \right]
\end{eqnarray*}
Projectivize to find the images of these vertices in the configuration space $W$ of the starting stage
\begin{eqnarray*}
\Gamma Q(12) &=& \left[ \begin{array}{c} 1/2(n+2)\\(n+1)/2(n+2)\\1/2\\0 \end{array}\right], \Gamma Q(13) = \left[\begin{array}{c}
1/2(n+3)\\(n+2)/2(n+3)\\1/2\\0 \end{array}\right]\\
\Gamma Q(34) &=& \left[ \begin{array}{c} 0\\1/2\\(n+1)/2(n+2)\\ 1/2(n+2)\end{array}\right], \Gamma Q(24) = \left[\begin{array}{c}
0\\1/2\\n/2(n+1)\\1/2(n+1)\end{array}\right]
\end{eqnarray*}

\begin{figure}[htb]
\begin{center}
\ \psfig{file=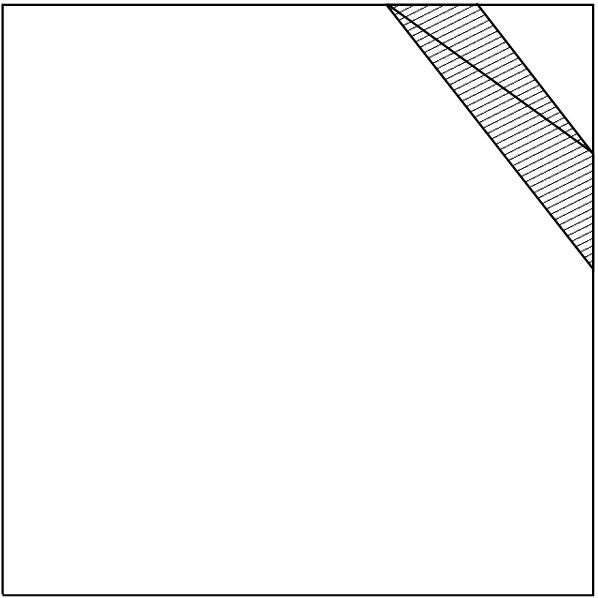, height=2truein} \caption{Projection of W(Q)} \label{project}
\end{center}
\setlength{\unitlength}{1in}
\begin{picture}(0,0)(0,0)
\put(1.1,0.6){$e_{24}$} \put(-1.2,0.6){$e_{14}$} \put(-1.2,2.5){$e_{13}$} \put(1.1,2.5){$e_{23}$} \put(-0.05,2.65){$\Gamma Q(12)$}
\put(0.6,2.65){$\Gamma Q(13)$} \put(1.1,2.05){$\Gamma Q(34)$} \put(1.1,1.55){$\Gamma Q(24)$}
\end{picture}
\end{figure}
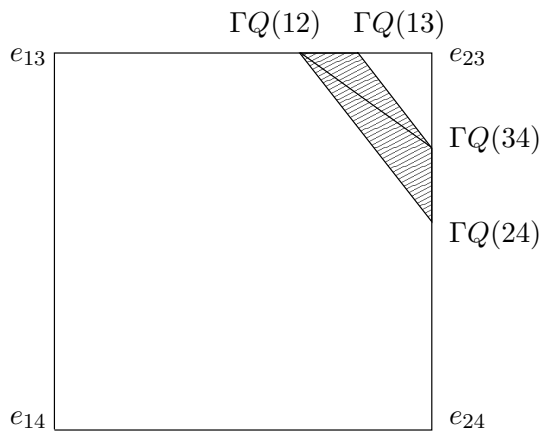
In terms of the vertices of the quadrilateral $W$, we get the linear combinations
\begin{eqnarray*}
\Gamma Q(12) &=& \left(\frac{1}{n+2}\right) e_{13} + \left(\frac{n+1}{n+2}\right) e_{23} \\
\Gamma Q(13) &=& \left(\frac{1}{n+3} \right)e_{13} + \left(\frac{n+2}{n+3}\right) e_{23} \\
\Gamma Q(34) &=& \left(\frac{n+1}{n+2}\right)e_{23} +\left(\frac{1}{n+2} \right)e_{24}\\
\Gamma Q(24) &=& \left(\frac{n}{n+1}\right)e_{23} + \left(\frac{1}{n+1}\right) e_{24}
\end{eqnarray*}
By symmetry, the quadrilateral $W$ is a square with side length $1/\sqrt{2}$. $W(Q)$ projects down as shown in the Figure~\ref{project}. From this we calculate the area of the quadrilateral $\Gamma Q(12) \Gamma Q(24) \Gamma Q(34) \Gamma Q(13)$ to be $1/2(n+1)(n+2)(n+3)$.

Now suppose that in the next split band 1 splits band 2. As shown in Figure~\ref{split}, the split is represented by the smaller simplex with vertices $Q(2), Q(3), Q(4)$ and $Q(12)$. The part of $W(Q)$ that lies inside this smaller simplex is the part of $W(Q)$ that lies above the shaded plane in Figure~\ref{split}. This is the triangle $Q(12) Q(24) Q(34)$. Projecting the triangle down to $W$ gives the triangle $\Gamma Q(12) \Gamma Q(24) \Gamma Q(34)$ in Figure~\ref{project}. The area of this triangle is $1/4(n+1)(n+2)^2$.

\begin{figure}[htb]
\begin{center}
\ \psfig{file=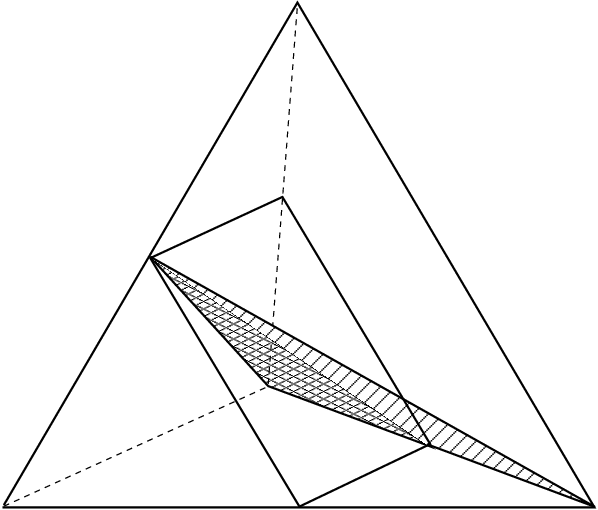, height=1.7truein} \caption{Band 1 splits Band 2} \label{split}
\end{center}
\setlength{\unitlength}{1in}
\begin{picture}(0,0)(0,0)
\put(-1.2,0.4){$Q(1)$} \put(0,0.4){$Q(13)$} \put(1,0.4){$Q(3)$} \put(0,2.3){$Q(2)$} \put(-0.15,1.7){$Q(24)$} \put(0.3,1.05){$Q(34)$}
\put(-0.95,1.4){$Q(12)$}

\end{picture}
\end{figure}
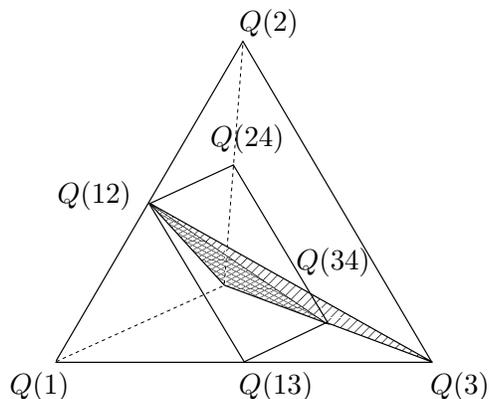

Hence, the probability of the split is the ratio of the area of the triangle $\Gamma Q(12) \Gamma Q(24) \Gamma Q(34)$ to the area of quadrilateral $\Gamma Q(12) \Gamma Q(24) \Gamma Q(34) \Gamma Q(13)$ viz. $(n+3)/2(n+2)$. We also see that $\vert Q(1) \vert$ increases by a factor of $2(n+2)/3$ in the split, finishing the computations.

The second issue is that for non-classical interval exchanges, isolated blocks are possible when there is a combinatorial reduction. They arise as follows: Write the combinatorial reduction as a concatenation of two non-classical interval exchanges with fewer bands, and consider the non-classical interval exchange on the right. If there is a splitting sequence of this smaller exchange in which every column moves every other column, then this splitting sequence is an isolated block for the original non-classical interval exchange. This point is overlooked in Kerckhoff \cite{Ker} and also in Proposition 3.6 of \cite{Nog} and Lemma 1.7 of \cite{Dan-Nog}. However, irreducibility of a non-classical interval exchange implies that one splits out of these blocks with a definite probability. See Step 1 of Proposition~\ref{inductive-step}.

From the next section on, we begin the technical details starting with the analysis of the Jacobian of the restriction of the projective linear map to the configuration spaces.

\section{Jacobian of the restriction}\label{Jacobian}

Let $\pi_0 \to \pi_1 \to \cdots \to \pi_n$ be a stage. For the rest of the paper, we simplify the notation by dropping the subscripts i.e., henceforth we will denote $\pi_n$ by $\pi$, the configuration space $W_n$ defined by $\pi_n$ as $W$ and the associated matrix $Q_n$ as $Q$. In this section, we analyze the Jacobian of $\Gamma Q$ restricted as a map from $W$ to $W_0$. The expression for the Jacobian of full map $\Gamma Q$ from $\Delta$ to itself is well known. See \cite{Buf}. At a point in the configuration space $W$, we write down a matrix for the derivative of the full map $\Gamma Q$ with respect to suitable decompositions as direct sums, of the tangent spaces to $\Delta$ at the point itself and its image under $\Gamma Q$. Then using the expression for the full Jacobian and the particular form of this matrix, we get an expression for the Jacobian of $\Gamma Q$ as a map from $W$ to $W_0$.

We fix some terminology. Given an affine subspace $L$ of $\R^\A$ and some $\y \in \R^\A$ which need not be in $L$, we let $T_\y L$ be the subspace of the tangent space at $\y \in \R^\A$ parallel to $L$. We will denote a vector in $\R^\A$ and also in the tangent space at any point $\y$, by the same letter $\mathbf{u}$, whenever the context is clear. Associated to a codimension 1 subspace $L$ transverse to a vector $\mathbf{u}$, there is a projection map $\phi_{\mathbf{u}}:T_{\y} \R^\A \to T_{\y} L$ by projecting along lines parallel to $\mathbf{u}$ till one hits $L$. A dilation of $\R^\A$ by $t>0$ shall be denoted by $\rho_t$. The derivative of a linear map $A$ acting on tangent spaces is denoted by $A$ itself.

Let $\Delta(Q)$ be the simplex with vertices the columns of $Q$. The map $\Gamma Q: \Delta \to \Delta$ is a diffeomorphism onto its image. This means that, for any point $\y \in \Delta$, the vector $\w = Q \y$ is transverse to both $\Delta(Q)$ and $\Delta$. Hence at the level of tangent spaces $\w$ is transverse to $T_{\w} \Delta(Q)$ and $T_{\w} \Delta$. So the projection map $\phi_{\w}: T_{\w} \R^\A \to T_{\w}\Delta$ by projecting along lines parallel to $\w$ is an isomorphism from $T_{\w} \Delta(Q)$ to $T_{\w} \Delta$. The full derivative $\dr \Gamma Q(\y): T_\y \Delta \to
T_{\Gamma Q(\y)} \Delta$ is the composition
\[
T_\y \Delta \xrightarrow{Q} T_{\w}\Delta(Q) \xrightarrow{\phi_{\w}} T_{\w} \Delta \xrightarrow{\rho_{|\w|^{-1}}} T_{\Gamma Q(\y)} \Delta
\]
It is known that (see Bufetov \cite{Buf}) the Jacobian $\J_\Delta(\Gamma Q)$ of the above composition is given by
\begin{equation}\label{full-jaco}
\J_\Delta(\Gamma Q)(\y) = \frac{1}{|\w|^d} = \frac{1}{|Q \y|^d}
\end{equation}
Fix unit vectors $\m_0$ and $\m$ normal to $W_0$ and $W$ respectively such that the tangent bundle $T\Delta$ has the orthogonal decompositions
\[
T\Delta = (\R \m_0) \oplus T W_0 \hspace{5mm}, \hspace{5mm} T\Delta = (\R \m) \oplus T W
\]
where $T W_0$ and $TW$ are the tangent bundles of the configuration spaces.

Since $\Gamma Q$ restricted to $W$ maps it into $W_0$ and is a diffeomorphism onto its image, at any point $\y \in W$ the full derivative $\dr \Gamma Q$ restricts to an isomorphism from $T_\y W$ to $T_{\Gamma Q(\y)}W_0$. So with respect to the orthogonal decompositions fixed above, the matrix at $\y$ for $\dr \Gamma Q$ has the form:
\begin{equation*}
\left(
\begin{array}{ccc|c}
 & & & * \\
 & \dr \Gamma Q|_{T_\y W \to T_{\Gamma Q(\y)}W_0} & & *\\
 & & & *\\
0 & \dotsc & 0 & c(\y)
\end{array}
\right)
\end{equation*}
From the matrix form we get
\begin{equation}\label{relate-jacobians}
\J_{\Delta}(\Gamma Q)(\y) = c(\y) \J(\Gamma Q)(\y)
\end{equation}
To compute $c(\y)$, we use the fact that
\[
\dr \Gamma Q(\y) = \rho_{|\w|^{-1}}\circ \phi_{\w} \circ Q
\]{
In the composition, we project the vector $Q \m \in T_\w \Delta(Q)$ by lines parallel to $\w$ to get $\phi_{\w}(Q \m) \in T_{\w} \Delta$. Next we write $\phi_{\w}(Q \m)$ as the unique linear combination
\[
\phi_{\w}(Q \m)= a(\y) \m_0 + \mathbf{p}
\]
in $T_\w \Delta$, where $\mathbf{p} \in T_\w W_0$. Last we apply the dilation $\rho_{|\w|^{-1}}$ to get $c(\y) = a(\y)/|\w|$. Hence from equation~\eqref{relate-jacobians} we get
\begin{equation}\label{Jac}
\J(\Gamma Q)(\y) = \frac{1}{a(\y) |\w|^{d-1}} = \frac{1}{a(\y)|Q \y|^{d-1}}
\end{equation}

\begin{figure}[htb]
\begin{center}
\ \psfig{file=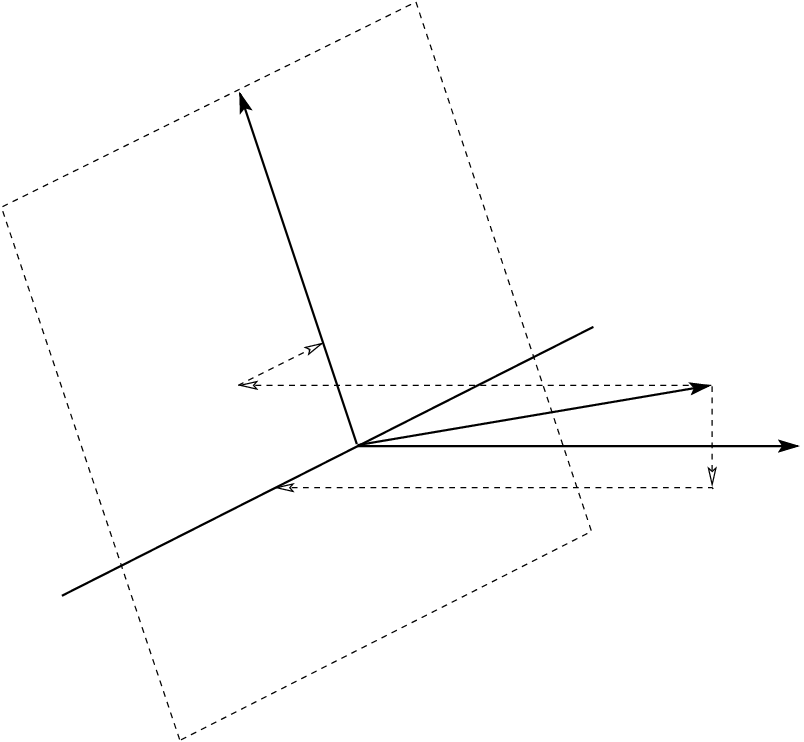, height=3truein}\caption{Schematic picture} \label{scheme-jaco}
\end{center}
\setlength{\unitlength}{1in}
\begin{picture}(0,0)(0,0)
\put(-1.9,1.1){$T_\w W_0$} \put(1.7,1.8){$\w$} \put(1.3,2.1){$Q\m$} \put(-0.9,3.4){$\m_0$} \put(-1.3,2.1){$\phi_\w(Q \m)$} \put(0.4,3){$T_\w
\Delta$}
\end{picture}
\end{figure}

\begin{lemma}
As $\y$ varies over $W$, the quantity $a(\y)$ remains constant.
\end{lemma}
\begin{proof}
To prove the lemma, we need to understand the projection $\phi_\w: T_\w \Delta(Q) \to T_\w \Delta$. We refer to Figure~\ref{scheme-jaco} which is a schematic picture of $T_\w \R^\A$. In the figure, the projection $\phi_\w$ by lines parallel to $\w$ is represented by the horizontal dashed arrows. Let $L$ be the codimension 1 subspace in $T_\w \R^\A$ spanned by $\w$ and $T_\w W_0$. If $A_0$ denotes the subspace of $\R^\A$ spanned by $W_0$ i.e.~the subspace of $\R^\A$ satisfying the equation that defines $W_0$, then $L = T_\w A_0$. So $L$ does not depend on $\w$.

Since $\phi_\w$ gives an isomorphism from $T_\w W(Q)$ to $T_\w W_0$, $T_\w W(Q)$ is in $L$. In the figure, we represent $L$ by the horizontal $(x,y)$-plane. The vector $\m_0$ is orthogonal to $T_\w W_0$ and transverse to $\w$, so in fact transverse to $L$.

For the remainder of the proof, we shall denote the $(d-1)$-dimensional subspace of $\R^\A$ that is parallel to $\Delta$ by $\Delta$ itself. The essential point from the discussion above is: thinking of each tangent space $T_\w \R^\A$ as $\R^\A$ itself, there are two $(d-1)$-dimensional transverse subspaces $L$ and $\Delta$ with a vector $\m_0 \in \Delta$ not in $L$. Moreover, as $\y$ varies over $W$ which is equivalent to saying as $\w$ varies over $W(Q)$, the vector $\w$ is constrained to lie in $L$. Any vector $\mathbf{v} \in T_\w \R^\A$ can be uniquely written as
\[
\mathbf{v} = a \m_0 + \mathbf{q}
\]
where $\mathbf{q} \in L$. Applying the projection $\phi_\w: T_\w \R^\A \to T_\w \Delta$
\[
\phi_\w(\mathbf{v}) = a \phi_\w(\m_0) + \phi_\w(\mathbf{q}) = a \m_0 + \phi_\w(\mathbf{q})
\]
The projection $\phi_\w$ has the form $\phi_\w(\mathbf{v}) = \mathbf{v} - b \w$. This implies that the vector $\phi_\w(\mathbf{q})$ belongs to $L$. So the component of $\phi_\w(\mathbf{v})$ along $\m_0$ remains $a \m_0$, even as $\w$ varies. Choosing $\mathbf{v} = Q \m$ we are done.
\end{proof}

By the above lemma, we can drop the dependence of $a(\y)$ on $\y$ in the expression for the Jacobian and write it just as $a$. However the number $a$ does depend on the stage of the expansion. Observe that, up to the number $a$, the expression for the Jacobian of the restriction looks similar to the Jacobian of a determinant 1 projective linear map with non-negative entries in dimension $(d-1)$. We shall make this observation precise and use it to compute measures in Section~\ref{proof}.

Next, we give the definition for a set of $C$-distributed vectors.
\begin{definition}\label{C-distribution}
Let $C>1$. A set of vectors $\{ \mathbf{u}_1, \mathbf{u}_2, \dotsc, \mathbf{u}_k \}$ are said to be $C${\em-distributed} if
\[
\frac{1}{C} < \frac{|\mathbf{u}_i|}{|\mathbf{u}_j|} < C
\]
for all $i,j$ such that $1 \leq i, j \leq k$.
\end{definition}

We conclude from the previous lemma that a stage $\pi$ is $C$-uniformly distorted if and only if the vertices of $W(Q)$ are $C^{\frac{1}{d-1}}$-distributed. Moreover if the vertices of $\Delta(Q)$ are $C$-distributed, then the vertices of $W(Q)$ are also $C$-distributed. So, to show that for almost every $\x$, there exists a stage $\pi_{\x,n}$ that is $C$-uniformly distorted, it is enough to show that for almost every $\x$, there exists a stage $\pi_{\x,n}$ such that the columns of $Q_{\x,n}$ are $C$-distributed. The proof of this roughly follows Kerckhoff's original proof in the case of classical interval exchanges while weakening various hypothesis.

In the next section, we recall the measure theory for determinant 1 projective linear maps with non-negative entries, from the standard simplex into itself. Additionally, we evaluate measures of certain subsets of the standard simplex, which shall be useful later.

\section{Measure Theory}\label{measure-theory}

\subsection{Measure Theory of Projective linear maps:}
Let $Q: \R^\A \to \R^\A$ be a linear map with non-negative entries and determinant 1. Projectivize $Q$ to get the map $\Gamma Q: \Delta \to \Delta$. We recall 4.4 Lemma 1 from Yoccoz \cite{Yoc}:

\begin{lemma}\label{measure-plm}
\[
\frac{\ell(\Gamma Q(\Delta))}{\ell(\Delta)} = \frac{1}{\prod\limits_{\alpha = 1}^{d} \vert Q(\alpha) \vert}
\]
\end{lemma}

\begin{remark}
The formula for $\ell(\Gamma Q(\Delta))$ can be derived by integrating the expression~\ref{full-jaco} for $\J_{\Delta}(\Gamma Q)$ on $\Delta$.
\end{remark}

\noindent The next corollary does not need the assumption that the matrix $Q$ has determinant 1. This will be relevant in Step 5 of Proposition~\ref{measure-norm}.

\begin{corollary}\label{elementary}
Let $E$ be an elementary matrix with the off-diagonal $(\beta, \alpha)$ entry equal to $R>0$ and let $Q$ be a linear map with non-negative entries. Then
\[
\frac{\ell(\Gamma(QE)(\Delta))}{\ell(\Gamma Q(\Delta))} = \frac{\vert Q(\alpha) \vert}{\vert Q(\alpha) \vert+ R \vert Q(\beta) \vert}
\]
\end{corollary}
\begin{proof}
As both sides of the claimed equality are unchanged if we replace $Q$ by $tQ$, we can reduce to the case $\det(Q) =1$. By applying Lemma~\ref{measure-plm}, we get
\[
\frac{\ell(\Gamma(QE)(\Delta))}{\ell(\Gamma Q(\Delta))} = \frac{\prod\limits_{\gamma=1}^{d} \vert Q(\gamma)\vert}{\prod\limits_{\gamma=1}^{d} \vert QE(\gamma) \vert} = \frac{\vert Q(\alpha) \vert}{\vert QE(\alpha) \vert} = \frac{\vert Q(\alpha) \vert}{\vert Q(\alpha)\vert+ R \vert Q(\beta) \vert}
\]
\end{proof}

\subsection{Evaluating measures:}
We apply Corollary~\ref{elementary} to evaluate measures of certain subsets of $\Delta$ that we shall encounter in Section 9.

\subsection*{Wedges:} For a pair of distinct labels $\alpha, \beta \in \A$ and a non-negative constant $R$, let $\Delta_R(\alpha,\beta)$ be the set of points in $\Delta$ whose convex combination $\x = \sum \lambda_\gamma e_\gamma$ satisfies
\begin{equation}\label{wedge-defn}
\frac{\lambda_\beta}{\lambda_\alpha} \geq R
\end{equation}
\begin{wrapfigure}{r}{0.22\textwidth}
\begin{center}
\ \epsfig{file=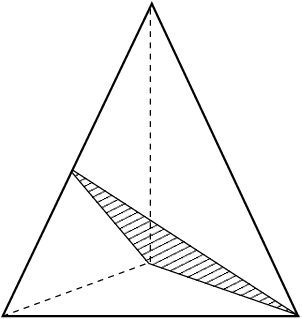, width=0.18\textwidth} \caption{$\Delta_R(\alpha,\beta)$}\label{wedge}
\end{center}
\setlength{\unitlength}{1in}
\begin{picture}(0,0)(0,0)
\put(0.7,2){$e_\alpha$} \put(0.05,0.62){$e_\beta$}
\end{picture}
\end{wrapfigure}
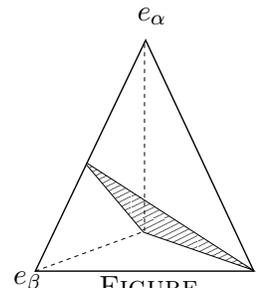
Figure~\ref{wedge} shows a schematic picture of $\Delta_R(\alpha,\beta)$. The shaded plane $F_R(\alpha, \beta)$ represents the set of points for which $\lambda_\beta /\lambda_\alpha = R$. The set $\Delta_R(\alpha,\beta)$ is the region in Figure~\ref{wedge} that lies below the shaded plane. Thus it looks like a {\it wedge} in $\Delta$.

Let $Q: \R^\A \to \R^\A$ be a linear map with non-negative entries. Write $R$ in the form
\begin{equation}\label{R}
R = K \frac{\vert Q(\alpha) \vert}{\vert Q(\beta)\vert}
\end{equation}
for some $K>0$. Then
\[
\Gamma Q(\Delta_R(\alpha,\beta)) = \Gamma (QE)(\Delta)
\]
where $E$ is the elementary matrix with the $(\beta, \alpha)$ entry equal to $R$. Applying Corollary~\ref{elementary}, we get
\begin{equation}\label{wedge-measure}
\frac{\ell \left( \Gamma Q(\Delta_R(\alpha,\beta)) \right)}{\ell \left( \Gamma Q(\Delta) \right)} = \frac{1}{1+K}
\end{equation}

\subsection*{Half-space intersections:} Having established Equation~\eqref{wedge-measure} for the relative measure of wedges, we move on to more general subsets of $\Delta$ which we call {\it thin half-space intersections}. The wedges $\Delta_R(\alpha, \beta)$ are an instance of these. We encounter thin half-space intersections in the proof of Proposition~\ref{measure-norm}.

A half-space $H$ in $\R^\A$ that intersects $\Delta$ and does not contain the vertex $e_\alpha$ shall be called a half-space {\it opposite} $\alpha$. Let $S_H$ denote the set of those $\gamma \neq \alpha$ for which the intersection $\partial H \cap [e_\alpha, e_\gamma]$ is non-empty, where $[e_\alpha,e_\gamma]$ denotes the edge of $\Delta$ joining $e_\alpha$ with $e_\gamma$. For $R_0>0$, a half-space $H$ opposite $\alpha$ is said to be {\em $R_0$-thin} if for every $\gamma \in S_H$, the intersection point $a_{\alpha \gamma}= \partial H \cap [e_\alpha, e_\gamma]$ has a convex combination that satisfies
\[
\frac{\lambda_\gamma}{\lambda_\alpha} \geq R_0
\]
We have the following inclusion
\begin{lemma}\label{halfspace}
For $R = R_0/d$, we have
\[
H \cap \Delta \subset \bigcup_{\gamma \in S_H} \Delta_{R}(\alpha, \gamma)
\]
\end{lemma}
\begin{proof}
Let $U$ be the sub-simplex of $\Delta$ whose vertices are indexed by $S_H \cup \{ \alpha \}$, and let $U(R_0)$ denote the set of points in $U$ whose convex combination in $U$ satisfies $\lambda_\alpha \leq 1/(1+R_0)$.
Let $\Delta(R_0)$ be the convex hull of $U(R_0)$ with the remaining vertices of $\Delta$.
\begin{claim}
\[
H \cap \Delta \subset \Delta(R_0)
\]
\end{claim}
\begin{proof}
Since both $H \cap \Delta$ and $\Delta(R_0)$ are convex sets, it is enough to show that every vertex of $H \cap \Delta$ is in $\Delta(R_0)$. By the virtue of being $R_0$-thin, the vertices $a_{\alpha \gamma}$ are in $\Delta(R_0)$. All the other vertices of $H \cap \Delta$ lie in $F(\alpha)$, the face of $\Delta$ opposite $e_\alpha$, which itself lies in $\Delta(R_0)$.
\end{proof}
\noindent The lemma then follows from the claim
\begin{claim}
\[
\Delta(R_0) \subset \bigcup_{\gamma \in S_H} \Delta_{R}(\alpha, \gamma)
\]
\end{claim}
\begin{proof}
Notice that the face $F(\alpha)$ of $\Delta$ sits in the union above. The vertices of $\Delta(R_0)$ that are not in $U$ lie in $F(\alpha)$. So it is enough to show that $U(R_0)$ lies in the union. The convex combination of a point in $U(R_0)$ satisfies $\lambda_\alpha \leq 1/(1+R_0)$. This implies that for at least one $\gamma \in S_H$,
\[
\lambda_\gamma \geq \left(\frac{1}{\# S_H}\right)\left( 1 - \frac{1}{1+R_0} \right) > \frac{R_0}{d(1+R_0)}
\]
It follows that
\[
\frac{\lambda_\gamma}{\lambda_\alpha} > \frac{R_0}{d}
\]
and so the point belongs to $\Delta_R(\alpha, \gamma)$.
\end{proof}
\end{proof}

\noindent In Proposition~\ref{measure-norm}, we need an upper bound on the measure of half-space intersections. This upper bound is derived there by using the inclusion of Lemma~\ref{halfspace} and then bounding the measures of the individual wedges in the union.

We now have all the preliminaries to carry out the proof of Theorem~\ref{Uniform-distortion}.

\section{Proof of Theorem~\ref{Uniform-distortion}}\label{proof}
Let $\pi$ be a stage in the expansion. Let $W$ be the configuration space at $\pi$ and let $Q$ be the matrix associated to $\pi$. 

\begin{proposition}\label{critical}
For almost every $\x \in \Gamma Q(W)$, every band $\alpha \in \A$ is split infinitely often and splits other bands infinitely often.
\end{proposition}

\begin{proof}
The first claim is that

\begin{claim}
Suppose in the expansion for some $\x$, a band $\alpha \in \A$ occurs in the critical positions infinitely often. Then $\alpha$ is split infinitely often and splits other bands infinitely often.
\end{claim}

\begin{proof}
Suppose $\alpha$ is split finitely many times. Then, there exists a future stage $\pi_{\x,m}$ after which $\alpha$ never gets split. So the actual width of $\alpha$ remains unchanged from $\pi_{\x,m}$ onwards. On the other hand, after $\pi_{\x,m}$, the band $\alpha$ splits some other band $\gamma$ infinitely often; this is impossible if the actual width of $\alpha$ is to remain fixed.

Alternatively, suppose $\alpha$ splits other bands finitely many times. A band leaves a critical position only if it splits some other band. Since $\alpha$ gets split infinitely often, there exists a future stage $\pi_{\x,m}$ after which $\alpha$ remains fixed in one of the critical positions and is the only band split thereafter. But then from $\pi_{\x,m}$ onwards, the actual widths of the rest of the bands remain unchanged because they never get split. This makes it impossible to split $\alpha$ ad infinitum.
\end{proof}

Thus, to prove Proposition~\ref{critical}, it is enough to show that almost surely, every band $\alpha$ occurs in the critical positions infinitely often. Let $Z$ be the set of those $\x$ in $\Gamma Q(W)$ for which there exists a band $\alpha$ that occurs in the critical positions finitely often in the expansion of $\x$. Fix such a point $\x$ and let $\A_1$ be the subset of those bands that occur in the critical positions finitely many times in the expansion of $\x$. Then there is a future stage $\pi_{\x,n}$ such that in the expansion of $\x$ after $\pi_{\x,n}$
\begin{enumerate}
\item The bands in $\A_1$ never occur in the critical positions.
\item Every band in $\A_2 = \A \setminus \A_1$ occur in the critical positions infinitely often.
\end{enumerate}
Moreover we assume that $\pi_{\x,n}$ is the first instance in the expansion of $\x$ in which the above properties hold.
\begin{claim}
$\A= \A_1 \sqcup \A_2$ is a combinatorial reduction of $\pi_{\x,n}$ and all future stages.
\end{claim}

\begin{proof}
Without loss of generality, suppose some $\beta \in \A_2$ occurs to the left of some $\alpha \in \A_1$ on the bottom. Since $\alpha$ never enters a critical position, whenever $\beta$ is split, the split has to begin at the other end of $\beta$. This increases the number of bands $k(\alpha)$ to the left of $\alpha$ on the bottom by 1. Moreover since $\beta$ is split infinitely often, the number $k(\alpha)$ has to become large enough to force $\alpha$ into the critical position on the bottom which is a contradiction. This proves the claim.
\end{proof}
Let $H$ be the set of finite splitting sequences $\jmath$ starting from $\pi$ such that final stage of $\jmath$ is combinatorially reducible. The set $H$ is countable because it is a subset of the countable set of finite splitting sequences.

By the above claim, we can define a map $\psi: Z \to H$. Let $Z_{\jmath} = \psi^{-1}(\jmath)$. If the set $Z_{\jmath}$ is non-empty, then the expansion after $\pi$ of each point in $Z_\jmath$ begins with the sequence $\jmath$. Since $H$ is a countable set, $\ell(Z) = \sum_{\jmath \in H} \ell(Z_\jmath)$. So to prove Proposition~\ref{critical} it is enough to show that $\ell(Z_{\jmath})=0$ for all $\jmath$ in $H$.
\begin{claim}
For every $\jmath$ in $H$, $\ell(Z_{\jmath}) = 0$.
\end{claim}
\begin{proof}
Let $\A= \A_1 \sqcup \A_2$ be the combinatorial reduction for the final stage $\pi'$ of $\jmath$. Let $(\lambda_\alpha ')$ denote the normalized widths at stage $\pi'$. We will show that a point in $Z_\jmath$ must satisfy constraint~\eqref{red-constraint} in the widths $(\lambda_\alpha ')$. Strong irreducibility implies that the set given by this constraint has measure zero, which then proves the claim.

For a point $\x$ that does not satisfy the constraint, set
\[
D = \sum_{\alpha \in \A_{2,-}} \lambda_{\alpha}' - \sum_{\beta \in \A_{2,+}} \lambda_{\beta}'
\]
and without loss of generality, assume that $D>0$. We claim that for a point with $D>0$, there is a stage in the expansion after $\pi'$ in which twice the sum of the actual widths of the bands in $\A_2$ is equal to $D$. When this happens, a band in $\A_1$ is forced into the critical position on the bottom. This would show that such a point cannot be in $Z_\jmath$.

Let $\lambda^\infty_\alpha$ denote the limit of the actual widths of $\alpha$ in the expansion of $\x$. To show that there is a stage in the expansion of $\x$ after $\pi'$ in which the sum of the actual widths of the bands in $\A_2$ is equal to $D$, we show that $\lambda^\infty_\alpha = 0$ for all $\alpha \in \A_2$. We prove this by showing that for any $\epsilon > 0$, there exists a future stage at which the actual widths of all bands in $\A_2$ are at most $\epsilon$.

For $\epsilon > 0$, there exists a stage $\pi_{\x,N}$ after $\pi'$ such that the actual widths $\lambda^{(N)}_\alpha$ at $\pi_{\x,N}$ satisfy $\lambda^\infty_\alpha < \lambda^{(N)}_\alpha < \lambda^\infty_\alpha + \epsilon$ for all $\alpha$ in $\A_2$. Then for every subsequent split after $\pi_{\x,N}$, the actual width of the band split can be reduced by at most $\epsilon$. This means that the actual width of a band doing the splitting is at most $\epsilon$. For the bands in $\A_2$, there is a future stage $\pi_{\x,N'}$ such that every band in $\A_2$ has split other bands at least once after $\pi_{\x,N}$ but before $\pi_{\x,N'}$. By the previous observation, the actual widths at $\pi_{\x,N'}$ of all bands in $\A_2$ are at most $\epsilon$ and we are done.
\end{proof}
This concludes the proof of Proposition~\ref{critical}.
\end{proof}
For $\alpha \in \A$ and $\x \in \Gamma Q(W)$, let $n(\x,\alpha)$ be the index of the first stage after $\pi$ such that, in the split $\pi_{n(\x,\alpha)} \to \pi_{n(\x,\alpha)+1}$ the band split is $\alpha$; if there is no such integer, set $n(\x,\alpha) = \infty$. From Proposition~\ref{critical}, we see that $n(\x,\alpha)$ is finite almost surely.

For $M>1$ let $ X_{M,\alpha} = \{ \x \in \Gamma Q(W) : \vert Q_{n(\x,\alpha)}(\alpha) \vert > M \vert Q(\alpha)\vert\}$. The set $X_{M,\alpha}$ is thus the set of those $\x$ in whose expansion the $\alpha$-column increases in norm by a factor greater than $M$ before $\alpha$ is split.

\begin{proposition}\label{measure-norm}
Let $\alpha$ be a band and $C>1$ be a constant such that
\[
\vert Q(\alpha) \vert > \frac{1}{C} \max_{\gamma \in \A} \vert Q(\gamma) \vert
\]
Then there exists a constant $M > 1$, depending only on $C$ and $d$, such that
\[
\ell(X_{M,\alpha}) < \frac{1}{2d} \ell(\Gamma Q(W))
\]
In fact, for all $M' > M$, the proportion of $X_{M',\alpha}$ in $\Gamma Q(W)$ has an upper bound that depends only on $M'$; moreover the bound $\to 0$ as $M' \to \infty$.
\end{proposition}

\begin{proof}
We shall prove Proposition~\ref{measure-norm} in a number of steps:

\subsection*{\it Step 1:} As a first step, we shall include the set $X_{M,\alpha}$ in a finite union of sets whose measures are easier to estimate.

Let $Y_{M,\alpha} = (\Gamma Q)^{-1}(X_{M,\alpha})$. For a $\beta \in \A$ and a positive constant $R$, recall from Section~\ref{measure-theory} that $\Delta_R(\alpha,\beta)$ is the set of points in $\Delta$ whose convex combination satisfies
\[
\frac{\lambda_\beta}{\lambda_\alpha} \geq R
\]
Write $R$ as
\[
R = K \frac{\vert Q(\alpha) \vert}{\vert Q(\beta) \vert}
\]
for some $K>0$ and let $Y_R(\alpha,\beta) = \Delta_R(\alpha,\beta) \cap W$ and let $X_R(\alpha,\beta) = \Gamma Q(Y_R(\alpha,\beta))$. We shall show that
\begin{lemma}
For $K= (M-1)/(d-1)$ we have the inclusion
\[
Y_{M,\alpha} \subseteq \bigcup_{\beta \neq \alpha} Y_R(\alpha,\beta)
\]
or equivalently
\[
X_{M,\alpha} \subseteq \bigcup_{\beta \neq \alpha} X_R(\alpha, \beta)
\]
\end{lemma}
\begin{proof}
Since $\alpha$ does not get split till $\pi_{n(\x,\alpha)}$, the column $Q_{n(\x,\alpha)}(\alpha)$ has the form
\[
Q_{n(\x,\alpha)}(\alpha) = Q(\alpha) + \sum_{\beta \neq \alpha} c_\beta Q(\beta)
\]
for non-negative integers $c_\beta$.
\begin{claim} \label{claim1}
Let $\y \in W$ be the pre-image of $\x$ under $\Gamma Q$. The convex combination $\y = \sum \lambda_\gamma e_\gamma$ must satisfy
\[
\frac{\lambda_\beta}{\lambda_\alpha } \geq c_\beta
\]
for all $\beta$.
\end{claim}
\begin{proof}
To simplify notation, let $m = n(\x,\alpha)$. Let $T$ denote the matrix associated to the splitting sequence $\pi \to \cdots \to \pi_m$. Starting
with the normalized widths $(\lambda_\gamma), \gamma \in \A$ at $\pi$, let $(\lambda^{(m)}_\gamma)$ denote the actual widths at $\pi_m$ resulting from the splitting sequence. Since $\alpha$ is never split till $\pi_m$, the width of $\alpha$ remains unchanged i.e.~$\lambda^{(m)}_\alpha = \lambda_\alpha$. Additionally, the entries of $\alpha$-th column of $T$ are given by $T_{\beta \alpha} = c_\beta$ for $\beta \neq \alpha$ and $T_{\alpha \alpha} = 1$. From
Section~\ref{dynamics}, the relationship between the old and the new widths is
\[
\lambda_\beta = \sum_{\gamma \in A} T_{\beta \gamma} \lambda^{(m)}_\gamma
\]
Since all terms on the right hand side above are non-negative, we get the inequality
\[
\lambda_\beta \geq T_{\beta \alpha} \lambda^{(m)}_\gamma = c_\beta \lambda_\alpha
\]
finishing the proof of the claim.
\end{proof}
\begin{claim} \label{claim2}
If $\vert Q_{n(\x,\alpha)}(\alpha) \vert \geq M \vert Q(\alpha) \vert$ for some constant $M>1$, then for at least one $\beta$
\[
c_\beta \geq \frac{M-1}{d-1} \frac{\vert Q(\alpha) \vert}{\vert Q(\beta) \vert}
\]
\end{claim}
\begin{proof}
By the additive property of the norm on $(\R_{\geq 0})^d$
\[
\vert Q_{n(\x,\alpha)}(\alpha) \vert = \vert Q(\alpha) \vert + \sum_{\beta \neq \alpha}c_\beta \vert Q(\beta) \vert
\]
If $\vert Q_{n(\x,\alpha)}(\alpha) \vert \geq M \vert Q(\alpha) \vert$, then
\[
\sum_{\beta \neq \alpha} c_\beta \vert Q(\beta) \vert \geq (M-1) \vert Q(\alpha) \vert
\]
which implies the claim.
\end{proof}
\noindent By Claims~\ref{claim1} and~\ref{claim2}, for each $\y \in Y_{M,\alpha}$, there is some $\beta$ such that the convex combination for $\y$ satisfies
\[
\frac{\lambda_\beta}{\lambda_\alpha} \geq K \frac{\vert Q(\alpha) \vert}{\vert Q(\beta) \vert}
\]
for $K = (M-1)/(d-1)$. This proves the lemma.
\end{proof}

To show that there is a choice of $M$ large enough such that $\ell(X_{M,\alpha})$ is bounded above by $(1/2d) \ell(\Gamma Q(W))$, it is enough to show that there is a $M$ large enough such that for every $\beta$, the quantity $\ell(X_R(\alpha, \beta))= \ell(\Gamma Q(Y_R(\alpha,\beta)))$ is bounded above by $(1/2d^2)\ell(\Gamma Q(W))$. So for the remainder of
this proof, we focus on one such $Y_R(\alpha,\beta)$.

\subsection*{\it Step 2:}
Figure~\ref{measurenorm} shows for some $\beta$, a schematic picture of the sets $\Delta_R(\alpha,\beta)$ and $Y_R(\alpha,\beta)$ inside $\Delta$. The simplex is drawn such that the vertex $e_\alpha$ is on the top and the opposite face $F(\alpha)$ is in the horizontal plane forming the floor of the simplex.
\begin{figure}[htb]
\begin{center}
\ \psfig{file=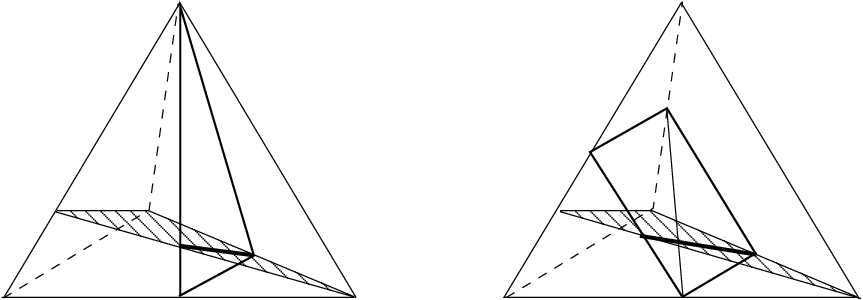, height=1.5truein} \caption{Schematic Picture} \label{measurenorm}
\end{center}
\setlength{\unitlength}{1in}
\begin{picture}(0,0)(0,0)
\put(-1.3,2.2){$e_\alpha$}\put(1.2,2.2){$e_\alpha$} \put(-2.2,0.5){$e_\beta$} \put(0.3,0.5){$e_\beta$} \put(-2.5,1.05){$e_R(\alpha,\beta)$}
\put(0,1.05){$e_R(\alpha,\beta)$}
\end{picture}
\end{figure}

The first picture is an instance when $\alpha$ is orientation preserving and the second picture is an instance when $\alpha$ is orientation reversing. Let $e_R(\alpha, \beta)$ be the point on the edge $[e_\alpha, e_\beta]$ whose convex combination satisfies $\lambda_\beta /\lambda_\alpha = R$. Let $F_R(\alpha,\beta)$ denote the convex hull of $e_R(\alpha, \beta)$ and the vertices of the simplex other than $e_\alpha$ and $e_\beta$. In the picture, it is represented by the plane that is shaded. Then $\Delta_R(\alpha,\beta)$ is the {\em wedge} in the simplex bounded by $F_R(\alpha,\beta)$ on the top and $F(\alpha)$ on the bottom i.e.~it is the region in the simplex below the shaded plane. As a result, $Y_R(\alpha,\beta)$ is the part of $W$ that lies below the intersection of the shaded plane and $W$, which is shown by the bold line segment in either picture.

The position of $e_R(\alpha,\beta)$ on the edge joining $e_\alpha$ and $e_\beta$ is determined by the value of $K$ and the ratio $\vert Q(\alpha) \vert / \vert Q(\beta) \vert$. The assumption $ \vert Q(\alpha) \vert > (1/C) \max_{\gamma \in \A} \vert Q(\gamma)\vert$ implies that the ratio $\vert Q(\alpha) \vert / \vert Q(\beta) \vert$ is bounded below by $1/C$. Using this lower bound, we see that the points in $\Delta_R(\alpha,\beta)$ and hence in $Y_R(\alpha,\beta)$ satisfy
\begin{equation}\label{thin}
\frac{\lambda_\beta}{\lambda_\alpha} \geq \frac{K}{C}
\end{equation}
This means that, for every $\beta$, the distance along the edge of the point $e_R(\alpha,\beta)$ from $e_\beta$ is bounded above by a quantity that depends only on $M$ and $C$. Moreover, this quantity goes to 0 as $M$ goes to infinity. In other words, each wedge is uniformly thin.

\subsection*{\it Step 3:} In this step, we triangulate $W$ by $(d-2)$-dimensional simplices $\Delta(i)$ in a specific way. Then to get the upper bound for measure relative to $\Gamma Q(W)$, it is enough to get the same bound for measure of the intersection with each simplex in the triangulation.

Recall from Section~\ref{dynamics}, each configuration space $W$ is a convex hull of a finite number of vertices, either of type $e_{\alpha \beta}$ or $e_\gamma$. Fix a triangulation of $W$ by $(d-2)$-dimensional simplices as follows:
\begin{enumerate}
\item Triangulate $\partial W$ without introducing new vertices. This can be done in any manner by adding an appropriate number of new diagonals.
\item Fix the vertex $p$ of $W$ where if $\alpha$ is orientation preserving then $p = e_\alpha$, and if $\alpha$ is orientation reversing then $p = e_{\alpha \gamma}$ for some $\gamma$. Cone off the triangulation of $\partial W$ to $p$ to get a triangulation of $W$.
\end{enumerate}
For example, in the second figure of Figure~\ref{measurenorm}, we have triangulated by adding in the diagonal of the quadrilateral. The number of simplices $\Delta(i)$ in a triangulation is bounded above by some number $n(d)$ that depends only on $d$.

Let the individual simplices in the triangulation of $W$ be $\Delta(i)$. By construction, each $\Delta(i)$ contains the vertex $p$ of $W$. To show the required bound for $\ell(\Gamma Q (Y_R(\alpha, \beta))$, it is enough to show that for $M$ large enough, the ratio $\ell \bigl(\Gamma Q(Y_R(\alpha,\beta) \cap \Delta(i))\bigr)/\ell\bigl(\Gamma Q(\Delta(i))\bigr)$ is bounded above by $1/2d^2$ for each $i$. Taking union over all $i$ implies the required bound for $\Gamma Q(Y_R(\alpha,\beta))$. So for the remainder of this proof, we will focus on one such $\Delta(i)$.

\subsection*{\it Step 4:} The intersection $Y_R(\alpha,\beta) \cap \Delta(i)$ is a half-space intersection $H \cap \Delta(i)$ with $\partial H = F_R(\alpha,\beta)$ (see Section~\ref{measure-theory}). In Inequality~\eqref{thin}, if $K/C > 1/2$, then the half-space intersection is opposite $p$. In this step, we shall show that $H \cap \Delta(i)$ opposite $p$ is $R/2$-thin. Hence we can use Lemma~\ref{halfspace} to include the half-space in a union of wedges in $\Delta(i)$. This reduces the task to showing a suitable upper bound for the ratio in each of these wedges.

The edges of $\Delta(i)$ with endpoint $p$ that have non-empty intersection with $F_R(\alpha,\beta)$ are the ones joining $p$ to a vertex $q$ of $\Delta(i)$ in $F(\alpha)$. Denote the set of vertices of $\Delta(i)$ in $F(\alpha)$ by $S_H$ and for each vertex $q \in S_H$, let $a_q$ denote the point of intersection $F_R(\alpha,\beta) \cap [p,q]$. Since $a_q \in F_R(\alpha, \beta)$, it's convex combination in the full simplex $\Delta$ satisfies
\begin{equation}\label{aqfull}
\frac{\lambda_\beta}{\lambda_\alpha} = R
\end{equation}
Let $a_q$ have the convex combination $a_q= \lambda_p p + \lambda_q q $ in $\Delta(i)$. Write
\begin{equation}\label{linearform}
p = \frac{e_\alpha + e_\gamma}{2} \hskip 10pt , \hskip 10pt q = \frac{e_\beta + e_\eta}{2}
\end{equation}
where $\gamma = \alpha$ or $\gamma \neq \alpha$ depending on whether $\alpha$ is orientation preserving or orientation reversing, and similarly $\eta= \beta$ or $\eta \neq \beta$. This means that $\lambda_\alpha = \lambda_p$ or $\lambda_\alpha= \lambda_p/2$, and similarly $\lambda_\beta = \lambda_q$ or $\lambda_\beta = \lambda_q/2$. In any case, combining this with Equation~\eqref{aqfull} implies that convex combinations in $\Delta(i)$ of the points on the line segment joining $a_q$ to $q$ satisfy
\[
\frac{\lambda_q}{\lambda_p} \geq \frac{R}{2}
\]
Thus our half-space intersection with $\Delta(i)$ opposite $p$ is $R/2$-thin. By applying Lemma~\ref{halfspace} to the intersection $Y_R(\alpha, \beta) \cap \Delta(i)$, we get for $R_0 = R/2(d-1)$, the inclusion
\begin{equation}\label{inclusion1}
Y_R(\alpha,\beta) \cap \Delta(i) \subset \bigcup_{q \in S_H} \Delta(i)_{R_0}(p,q)
\end{equation}
Consider the vectors $Qp, Qq \in \R^\A$. Using Equation~\eqref{linearform}, the ratio $\vert Qp \vert / \vert Qq \vert$ satisfies
\[
\frac{\vert Qp \vert}{\vert Qq \vert} =  \frac{\vert Q(\alpha) \vert+ \vert Q(\gamma)\vert}{ \vert Q(\beta)\vert+ \vert Q(\eta) \vert} < (C+1)\frac{\vert Q(\alpha) \vert}{\vert Q(\beta) \vert}
\]
So if we set
\[
R_1 = \left( \frac{K}{4dC} \right) \frac{\vert Qp \vert}{\vert Qq \vert}
\]
then $R_0 > R_1$. This implies the inclusion
\[
\Delta(i)_{R_0}(p,q) \subset \Delta(i)_{R_1}(p,q)
\]
for each $q \in S_H$. Combining it with inclusion~\eqref{inclusion1}, we get
\[
Y_R(\alpha,\beta) \cap \Delta(i) \subset \bigcup_{q \in S_H} \Delta(i)_{R_1}(p,q)
\]
The number of simplices $\Delta(i)$ is at most $n(d)$. So to establish the required upper bound on $\ell\bigl(\Gamma Q(Y_R(\alpha,\beta)\cap \Delta(i))\bigr)$, it is enough to show that there is a choice of $M$ and consequently of $R_1$ such that
\begin{equation} \label{estimate}
\ell\bigl(\Gamma Q(\Delta(i)_{R_1}(p,q))\bigr) \leq \frac{1}{2d^2 n(d)} \ell\bigl(\Gamma Q(\Delta(i))\bigr)
\end{equation}

\subsection*{\it Step 5:}  Equation~\eqref{wedge-measure} applies to projective linear maps with non-negative entries of the standard $(d-1)$-dimensional simplex into itself. So to apply Equation~\eqref{wedge-measure} to get Estimate~\eqref{estimate}, we need to extract such a map $\Gamma \widehat{Q}:\Delta_{d-2} \to \Delta_{d-2}$ from the restriction of $\Gamma Q$ to $\Delta(i)$. This creates two issues: First, to extract the map $\Gamma \widehat{Q}$, we have to identify $\Delta(i)$ with $\Delta_{d-2}$; this introduces a distortion in the standard measure on $\Delta_{d-2}$. So we need to show that this distortion is universally bounded. Second, we need to show that the Jacobian $\J(\Gamma Q)$ restricted to $\Delta(i)$ is, up to a universal constant, the same as the full Jacobian $\J_{\Delta_{d-2}}(\Gamma \widehat{Q})$ of the extracted map. In the final step, we show how to do this.

Recall from Section~\ref{dynamics} that the volume forms on configuration spaces are normalized so that the volume of each configuration space is 1. Hence, up to the ratios of the configuration space volumes, for a measurable subset $X$ of $\Gamma Q(W)$
\begin{equation}\label{change-var}
\ell(X) \approx \int_{W} \J(\Gamma Q)(\y) \chi_{Y}(\y) d \y
\end{equation}
where $\chi_Y$ is the characteristic function of the set $Y = (\Gamma Q)^{-1}(X)$. Using~\eqref{change-var}, we get
\[
\frac{\ell\bigl(\Gamma Q(\Delta(i)_{R_1}(p,q))\bigr)} {\ell\bigl(\Gamma Q(\Delta(i))\bigr)} = \frac{ \int\limits_{\Delta(i)_{R_1}(p,q)} \J(\Gamma Q)(\y) d
\y}{\int\limits_{\Delta(i)}\J(\Gamma Q)(\y) d \y}
\]
Recall from Section~\ref{Jacobian}, the expression for the Jacobian of the restriction to $W$.
\[
\J(\Gamma Q)(\y) = \frac{1}{a\vert Q\y \vert^{d-1}}
\]
So
\begin{equation}\label{ratio}
\frac{\ell\bigl(\Gamma Q(\Delta(i)_{R_1}(p,q))\bigr)} {\ell\bigl(\Gamma Q(\Delta(i))\bigr)} = \frac{\int\limits_{\Delta(i)_{R_1}(p,q)}
\frac{1}{\vert Q \y \vert^{d-1}} d\y}{\int\limits_{\Delta(i)} \frac{1}{\vert Q \y \vert^{d-1}} d\y}
\end{equation}
In Equation~\eqref{ratio}, the integrand looks similar to the Jacobian of a determinant 1 projective linear map, with non-negative entries, of a standard $(d-2)$-dimensional simplex $\Delta_{d-2}$. We indicate how to make this observation precise and then use it to bound the ratio from above.

Let $A_0$ be the subspace in $\R^\A$ spanned by $W_0$ and let $A$ be the subspace in $\R^\A$ spanned by $W$. Thus $A_0$ and $A$ are subspaces of $\R^\A$ satisfying the equations that define $W_0$ and $W$ respectively. Fix some linear isomorphism $F: \R^{d-1} \to A_0$ such that $F(\Delta_{d-2})$ contains $W_0$. Also, up to a permutation of the standard basis in $\R^{d-1}$, there is a unique linear isomorphism $G(i):\R^{d-1} \to A$ such that $G(i)(\Delta_{d-2}) = \Delta(i)$. Use the identification $G(i)$ to label the vertices of $\Delta_{d-2}$ by the corresponding labels of $\Delta(i)$.

Using the map $F$ and the identifications $G(i)$, we define a linear map $\widehat{Q} = F^{-1} \circ Q \circ G(i)$. Since the projectivization $\Gamma  \widehat{Q}$ maps $\Delta_{d-2}$ into itself, the map $\widehat{Q}$ has non-negative entries.

Since there are finitely many configuration spaces and finitely many simplices in the triangulation of each configuration space, the distortion by $G(i)$ of the standard Lebesgue measure on $\Delta_{d-2}$ is bounded above by a constant that depends only on $d$. This means that for a point $\y \in \Delta_{d-2}$, the quantity $1 / \vert QG(i)\y \vert^{d-1}$ (the integrand on the right in~\eqref{ratio}) is the same as the full Jacobian $\J_{\Delta_{d-2}}(\Gamma \widehat{Q})(\y) = 1 / \vert \widehat{Q}\y \vert^{d-1}$, up to a constant that depends only on $d$. This also means that up to a universal constant, the norm of the $r$-th column $\vert \widehat{Q}(r) \vert$ is the same as $\vert Qr \vert$.

The restriction $\Delta(i)_{R_1}(p,q)\to \Gamma Q(\Delta(i)_{R_1}(p,q))$ corresponds to $\Gamma (\widehat{Q}E):\Delta_{d-2} \to \Delta_{d-2}$ where $E$ is the elementary matrix whose off-diagonal entry in the $(q,p)$ place is $R_1$.

Applying Corollary~\ref{elementary}, we get
\begin{eqnarray*}
\frac{\ell\bigl(\Gamma Q (\Delta(i)_{R_1}(p,q))\bigr)}{\ell\bigl(\Gamma Q(\Delta(i))\bigr)} &\approx& \frac{ \vert \widehat{Q}(p) \vert}{ \vert \widehat{Q}(p) \vert + R_1 \vert \widehat{Q}(q) \vert}\\
&=& \vert \widehat{Q}(p) \vert \left[ \vert \widehat{Q}(p) \vert + \left( \frac{K}{4dC} \right) \frac{\vert Qp \vert \vert \widehat{Q}(q) \vert}{\vert Qq \vert}\right]^{-1}\\
&\approx& \frac{4dC}{4dC+K} < \frac{4dC}{K}
\end{eqnarray*}
Finally, if $M$ is large enough such that $K > 8d^3 n(d)C$, then we get Estimate~\eqref{estimate}; moreover it is easily checked that as $M \to \infty$, $K \to \infty$ and so the second part of the Proposition~\ref{measure-norm} also follows.
\end{proof}

In the next lemma, we use the technique in Step 5 of Proposition~\ref{measure-norm} to estimate from below, the probability of certain splits. This shall be useful later in Proposition~\ref{inductive-step}.

\begin{lemma}\label{split-estimate}
Let $\gamma$ and $\sigma$ be the bands in critical positions at some stage $\pi$. Let $C>1$ be a constant such that
\[
\vert Q(\gamma) \vert > \frac{1}{C} \max_{\alpha \in \A} \vert Q(\alpha) \vert
\]
Then up to a universal constant depending only on $d$, the probability that $\gamma$ splits $\sigma$ is bounded below by the quantity $1/n(d)(2C)^d $.
\end{lemma}
\begin{proof}
Let $\pi'$ be the generalized permutation that results from $\gamma$ splitting $\sigma$. Let $W$ and $W'$ denote the configuration spaces defined by $\pi$ and $\pi'$.

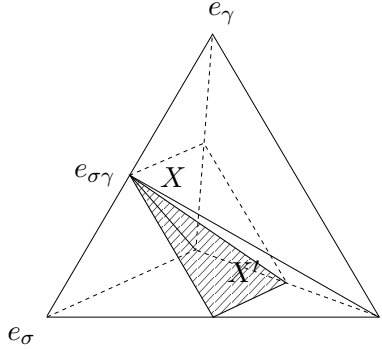
\begin{figure}[htb]
\begin{center}
\ \psfig{file=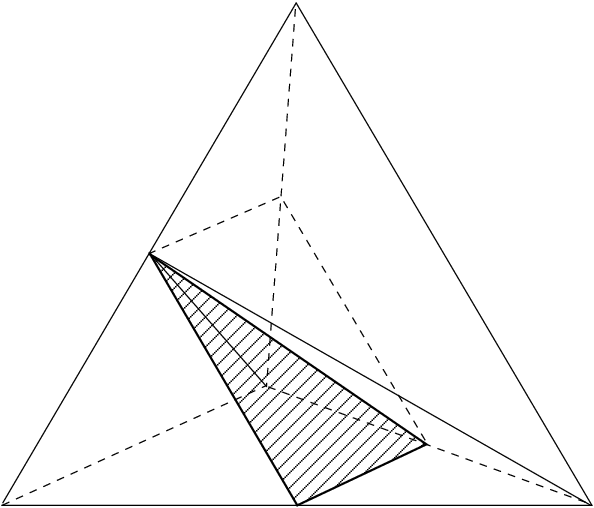, height=1.5truein} \caption{$\gamma$ splits $\sigma$} \label{scheme}
\end{center}
\setlength{\unitlength}{1in}
\begin{picture}(0,0)(0,0)
\put(0,2.2){$e_\gamma$} \put(-1.05,0.5){$e_\sigma$} \put(-0.25,1.3){$X$} \put(0.1,0.8){$X'$} \put(-0.7,1.35){$e_{\sigma \gamma}$}
\end{picture}
\end{figure}
A schematic picture is shown in Figure~\ref{scheme} with the shaded part $X'$ in $W$ representing the split. In other words, the shaded part $X'$ is the image of $W'$ under the projectivization of the elementary matrix associated to the split. Similarly, let $X$ be $W  \setminus X'$; this corresponds to $\sigma$ splitting $\gamma$ instead. Except when $\gamma$ is the only orientation reversing band on its side, $X$ is non-empty. When $X$ is empty, the split occurs with probability 1; so we focus on the split in which $X$ is non-empty. Let $Z$ denote $\partial X' \cap \partial X$ i.e.~the part of the boundary that separates the two sets.

The vertices of $W$ fall into the following categories:
\begin{enumerate}
\item If a vertex has the linear combination $(1/2)e_\sigma+ (1/2)e_\tau$ with $\tau \neq \gamma$, then it lies in $X' \setminus Z$.
\item If a vertex has the linear combination $(1/2)e_\gamma + (1/2)e_\tau$ with $\tau \neq \sigma$, then it lies in $X \setminus Z$.
\item All other vertices of $W$ lie in $Z$.
\end{enumerate}

Just as in Proposition~\ref{measure-norm}, let $\Delta(i)$ be the triangulation of $X$. Denote the set of vertices of $\Delta(i)$ by $V(i)$. Let $k$ be the index among all $i$, for which the measure $\ell\bigl(\Gamma Q (\Delta(k))\bigr)$ is the maximum. The triangulation has at most $n(d)$ simplices. Hence we get the estimate
\[
\ell\bigl(\Gamma Q (X)\bigr) \leq n(d) \ell\bigl(\Gamma Q (\Delta(k))\bigr) \\
\]
Let $V_Z$ be the subset of vertices of $\Delta(k)$ that lie in $Z$. Construct any triangulation of $X'$ such that there is a simplex $\Delta'(j)$ for which the points in $V_Z$ are vertices. Then
\[
\frac{\ell\bigl(\Gamma Q (X')\bigr)}{\ell\bigl(\Gamma Q(X)\bigr)} > \biggl(\frac{1}{n(d)}\biggr)\frac{\ell \bigl(\Gamma Q(\Delta'(j))\bigr)}{\ell \bigl(\Gamma Q(\Delta(k))\bigr)}
\]
So it is enough to show a lower bound for the right side in the equation above. Let $V'(j)$ denote the set of vertices of $\Delta'(j)$. Associated to the simplices $\Delta(k)$ and $\Delta'(j)$, we define maps $\widehat{Q_k}= F^{-1} \circ Q \circ G(k)$ and $\widehat{Q_j}= F^{-1}\circ Q \circ G'(j)$ from the standard $(d-2)$-dimensional simplex into itself. See Step 5 of Proposition~\ref{measure-norm} for the definition of these maps. The gist of the discussion in that step is that up to a universal constant introduced in the identifications,
\begin{equation}\label{hat}
\frac{\ell \left(\Gamma Q(\Delta'(j))\right)}{\ell \left(\Gamma Q(\Delta(k))\right)} \approx \frac{\ell \left(\Gamma \widehat{Q_j}(\Delta_{d-2})\right)}{\ell \left(\Gamma \widehat{Q_k}(\Delta_{d-2})\right)}
\end{equation}
There is a unique linear isomorphism $G(j k)$ of $\R^{d-1}$ such that $\widehat{Q_k} = \widehat{Q_j} \circ G(j k)$. Since the number of identification maps $G(*)$ is finite, the number of linear isomorphisms that relate them is also finite. Hence the ratio of the determinants of $\widehat{Q_*}$ is universally bounded. This bound can be incorporated in applying Lemma~\ref{measure-plm} to evaluate the right hand side in~\eqref{hat}. Recall that the norm of the $r$-th column $\widehat{Q_*}(r)$ is up to a universal constant the same as the norm of the vector $Qr$ in $\R^\A$. In conclusion,
\begin{equation}\label{hat2}
\frac{\ell \left(\Gamma \widehat{Q_j}(\Delta_{d-2})\right)}{\ell \left(\Gamma \widehat{Q_k}(\Delta_{d-2})\right)} \approx \frac{\prod\limits_{p \in V(k)} \vert Q p \vert}{\prod\limits_{q \in V'(j)} \vert Q q \vert} = \frac{\prod\limits_{p \in V(k) \setminus V_Z} \vert Q p \vert}{\prod\limits_{q \in V'(j) \setminus V_Z} \vert Q q \vert}
\end{equation}
In~\eqref{hat2}, the number of vertices in the numerator is exactly equal to the number of vertices in the denominator. So the terms in the products in the numerator and the denominator pair off. In addition, any vertex $p$ in $V(k) \setminus V_Z$ is a vertex in $X \setminus Z$ i.e.~in our categorization, the vertex belongs to category (2). Hence it's linear combination is $p = (1/2) e_\gamma+ (1/2) e_\tau$, and so it satisfies $\vert Q p \vert > \vert Q (\gamma) \vert/2$. Hence for a vertex $q$ in $V'(j) \setminus V_Z$, we have the estimate
\[
\frac{\vert Q p \vert}{\vert Q q\vert} > \frac{\vert Q(\gamma) \vert /2}{\max_{\alpha \in  \A} \vert Q(\alpha)
\vert} > \frac{1}{2C}
\]
Using the estimate in~\eqref{hat2} gives the lower bound
\[
\frac{\ell \left(\Gamma \widehat{Q_j}(\Delta_{d-2})\right)}{\ell \left(\Gamma \widehat{Q_k}(\Delta_{d-2})\right)} \succ \biggl( \frac{1}{2C} \biggr)^{d-1}
\]
Thus from~\eqref{hat} we get the lower bound for the probability of the split
\[
\frac{\ell\bigl(\Gamma Q (X)\bigr)}{\ell\bigl(\Gamma Q(W)\bigr)} \succ \frac{1}{1+n(d)(2C)^{d-1}} > \frac{1}{n(d) (2C)^d}
\]
\end{proof}

The next proposition, which goes back to Kerckhoff \cite{Ker}, has the following idea: One starts off by arranging the columns $Q(\alpha)$ in the order of decreasing norm and then picks out a $C_0$-distributed subset $\A_0$ starting with the column with the largest norm and following the order. Then the proposition shows that there is a definite probability, independent of the stage, that some future stage has a larger subset $\A_1 \supsetneq \A_0$ that is $C_1$-distributed and contains the column with the largest norm, with the constant $C_1$ independent of the stage. Thus the proposition can be used as an iterative step which we iterate over a finite number of times to get, with a definite probability, a future stage that is $C$-distributed. The precise statement of the proposition here is identical to 4.4 Lemma 5 in Yoccoz \cite{Yoc}, except that it is formulated for non-classical interval exchanges.

\begin{proposition}\label{inductive-step}
Let $C_0 > 1$ be a constant and let $\A_0 \subsetneq \A$ be a set of $C_0$-distributed columns which contains the largest column in norm. Then there exists constants $c_1 \in (0,1)$ and $C_1 > 1$, depending only on $C_0$ and $d$, and a finite set of future stages $\pi_\theta$, after $\pi$, that satisfy the following two conditions:
\begin{enumerate}
\item Denote the matrix corresponding to each $\pi_\theta$ by $Q_\theta$ and the configuration space by $W_\theta$. The sets $\Gamma Q_\theta(W_\theta)$ have disjoint interiors and
\[
\sum_\theta \ell(\Gamma Q_\theta(W_\theta)) \geq c_1 \ell(\Gamma Q(W))
\]
\item For every $\theta$, there exists a set of $C_1$-distributed columns $\A_\theta \supsetneq \A_0$ that contains the largest column at the stage $\pi_\theta$.
\end{enumerate}
\end{proposition}

\begin{proof}
We recall the basic idea of the proof from Section~\ref{example}: Before the norms of columns in $\A_0$ increase by some factor, suppose one of the following happens: either a column outside $\A_0$ becomes the column with the largest norm or a column in $\A_0$ moves a column outside. At this point, if we let $\A_1$ be $\A_0$ union this outside column, then as we shall see, the new collection is $C_1$-distributed where $C_1$ depends only on the initial constants. Then it remains to show that with a definite probability, one of the two events happens.

The proof here follows the proof of Lemma 5 of Section 4.4 in Yoccoz \cite{Yoc} closely, except that the individual steps require justification for non-classical interval exchanges. The steps are as follows:

\subsection*{\it Step 1:} To make it possible for a column in $\A_0$ to move a column outside, a column outside $\A_0$ needs to land in one of the critical positions. In this step, we show that there is a definite probability that this happens before the maximum of the norms of the columns increases by a definite factor. Precisely, we claim that $\pi$ can be split to a future stage $\pi_\iota$ such that, for constants $c_1' \in (0,1)$ and $C_1' > 1$ depending only on $C_0$ and $d$, we have
\begin{equation}\label{upper}
\max_{\alpha \in \A} \vert Q_\iota(\alpha) \vert < C_1' \max_{\alpha \in \A} \vert Q(\alpha) \vert
\end{equation}
\begin{equation}\label{lower}
\ell(\Gamma Q_\iota(W_\iota)) > c_1' \ell(\Gamma Q(W))
\end{equation}
and at least one of the bands in the critical positions in $\pi_\iota$ does not belong to $\A_0$.
\begin{proof}
By Proposition~\ref{critical}, for each node $\pi$ in the attractor and for each band $\beta$ in $\A$, there are splitting sequences which bring an end of $\beta$ to a critical position. For each pair $(\pi, \beta)$, choose one such splitting sequence. The number of chosen sequences is equal to $d$ times the number of nodes in the attractor. This means that there is an upper bound $h$, depending only on $d$, on the lengths of all chosen sequences.

For some band $\beta$ in $S \setminus \A_0$, let $\iota(\beta)$ be the chosen splitting sequence starting from $\pi$ that brings $\beta$ into one of the critical positions. Let $\iota$ be the shortest prefix of $\iota(\beta)$, for which a band not in $\A_0$ is in a critical position.

By the above discussion, the length of $\iota$ is at most $h$. This implies that the norm of a column in $\A_0$ participating in $\iota$ can increase by a factor of at most $2^h$. So we get the estimate
\[
\max_{\alpha \in \A_0} \vert Q_\iota(\alpha) \vert < 2^h \max_{\alpha \in \A} \vert Q(\alpha) \vert
\]
Thus~\eqref{upper} holds with $C_1' = 2^h$.

To show the bound~\eqref{lower} for the measures, it is enough to show that every split in $\iota$ has relative probability bounded below by a constant that depends only on $C_0$ and $d$.

Suppose we are at some intermediate stage $\pi_1$ of $\iota$. Following $\pi_1$, suppose a band $\gamma$ in $\A_0$ splits a band $\sigma$ in $\A_0$ to give us the next stage $\pi_2$. The column $Q_1(\gamma)$ satisfies the estimate
\[
\vert Q_1(\gamma) \vert > \vert Q(\gamma) \vert > \frac{1}{C_0}\max_{\alpha \in \A} \vert Q(\alpha) \vert > \frac{1}{C_0C_1'} \max_{\alpha \in \A} \vert Q_1(\alpha) \vert
\]
By Lemma~\ref{split-estimate}, the probability of this split is bounded below by $1/n(d)(2C_0C_1')^d$, up to a universal constant that depends only on $d$. Finally, using the fact that the length of $\iota$ is bounded above by $h$, we get
\[
\ell\bigl(\Gamma Q_\iota(W_\iota)\bigr) \succ \frac{1}{n(d)^h(2C_0C_1')^{dh}}\ell\bigl(\Gamma Q(W)\bigr)
\]
So setting
\[
c_1' \approx \frac{1}{n(d)^h(2C_0C_1')^{dh}}
\]
proves Step 1.
\end{proof}

\subsection*{\it Step 2:} Notice that for every $\alpha \in \A_0$, we have $ \vert Q_\iota(\alpha) \vert \geq (1/C_1')\max_{\gamma \in \A} \vert Q_\iota(\gamma) \vert$. So we can apply Proposition~\ref{measure-norm} to the stage $\pi_\iota$. Let
\[
X_{M,\alpha} = \{ \x \in \Gamma Q_\iota(W_\iota): \vert Q_{n(\x,\alpha)}(\alpha) \vert > M \vert Q_\iota(\alpha) \vert \}
\]
By Proposition~\ref{measure-norm}, there exists $M > 1$, depending on $C_1'$ and $d$ such that $\ell(X_{M,\alpha}) < (1/2d) \ell(\Gamma Q_\iota(W_\iota))$. If we let
\[
X = \Gamma Q_\iota(W_\iota) \setminus \bigcup_{\alpha \in \A_0} X_{M,\alpha}
\]
then we have the estimate
\begin{equation}\label{E}
\frac{\ell(X)}{\ell(\Gamma Q_\iota(W_\iota))} \geq 1 - \sum_{\alpha \in \A_0} \frac{\ell(X_{M,\alpha})}{\ell(\Gamma Q_\iota(W_\iota))} > 1 - \sum_{\alpha \in \A_0} \frac{1}{2d} > \frac{1}{2}
\end{equation}
From this point on, the rest of the proof follows Yoccoz \cite{Yoc} almost verbatim. We include it here for completeness.

\subsection*{\it Step 3:} Since at $\pi_\iota$, a band outside $\A_0$ is in a critical position, it is not possible to get a stage in which both bands in the critical positions are in $\A_0$ unless a column in $\A_0$ moves a column outside. This observation can be exploited to show that with a definite probability, either a band outside $\A_0$ becomes large enough in norm or a band in $\A_0$ moves a band outside.

For each $\x \in X$, let $\pi_{\x,m}$ denote the stage after $\pi_\iota$ such that $\pi_{\x,m-1} \to \pi_{\x,m}$ is the first instance when a band in $\A_0$ is split. To each $\x \in X$, we associate a stage $\pi(\x)$ in the path $\pi_\iota \to \cdots \to \pi_{\x,m}$ as follows:
\subsection*{\it Case (1):} If at some intermediate stage $\pi'$
\begin{equation}\label{e1}
\vert Q'(\beta) \vert \geq \max_{\alpha \in \A} \vert Q_\iota(\alpha) \vert
\end{equation}
for some $\beta \in \A \setminus \A_0$, then we let $\pi(\x)$ be the first instance when Inequality~\eqref{e1} is true for some band in $S \setminus \A_0$. This also means that in $\pi(\x)$, there is a unique band $\beta \in \A \setminus \A_0$ that satisfies Inequality~\eqref{e1}.
\subsection*{\it Case (2):} If no band in $S \setminus \A_0$ satisfies Inequality~\eqref{e1} at any intermediate stage, we set $\pi(\x) = \pi_{\x,m}$. Here a band $\beta \in \A \setminus \A_0$ splits a band $\alpha \in \A_0$ in the final split before $\pi(\x)$. So
\begin{equation}\label{e2}
\vert Q_\theta(\beta) \vert = \vert Q_{\theta-1}(\beta) \vert + \vert Q_{\theta-1}(\alpha) \vert
\end{equation}

\subsection*{\it Step 4:} Given the estimate~\eqref{E}, it is possible to select a finite number of stages $\pi_\theta$ from the collection of stages $\pi(\x)$ constructed in Step 3 such that $\Gamma Q_\theta(W_\theta)$ have disjoint interiors and
\[
\ell \left( \cup_\theta \Gamma Q_\theta(W_\theta) \right) > \frac{1}{2} \ell(X)> \frac{1}{4} \ell( \Gamma Q_\iota(W_\iota))
\]
For each $\pi_\theta$ as above, let $\A_\theta$ be the set of $\gamma \in \A$ that satisfy
\begin{equation}\label{e3}
\vert Q_\theta(\gamma) \vert > C_0^{-1} \max_{\alpha \in \A_0} \vert Q(\alpha) \vert
\end{equation}

\begin{claim}
For each $\theta$, the collection $\A_\theta$ is strictly larger than $\A_0$ and is $C_0C_1'(1+M)$-distributed.
\end{claim}
\begin{proof}
The stage $\pi_\theta$ belongs to one of the two cases in Step 3. In both cases, the distinguished band $\beta$ is in $\A_\theta \setminus \A_0$.

Given the lower bound~\eqref{e3} in the definition of $\A_\theta$, it is enough to show that the largest column in norm has an appropriate upper bound. Because we are dealing with subsets of $X$, every band $\alpha \in \A_0$ satisfies
\begin{equation}\label{e4}
\vert Q_\theta(\alpha) \vert < M \vert Q_\iota(\alpha) \vert
\end{equation}
This implies
\begin{equation}\label{e5}
\max_{\alpha \in \A_0} \vert Q_\theta(\alpha) \vert < M \max_{\alpha \in \A_0} \vert Q_\iota(\alpha) \vert
\end{equation}
If the largest column at $\pi_\theta$ is in $\A_0$, then estimates~\eqref{upper},~\eqref{e3} and~\eqref{e5} imply that $\A_\theta$ is $C_0C_1'M$-distributed. So it is also $C_0C_1'(1+M)$-distributed and we are done. Hence to finish the proof, we shall assume that the largest column is not in $\A_0$. A stage $\pi_\theta$ can belong to Case (1) or Case (2) of Step 3. We argue the two cases separately:

\subsection*{\it Case 1:} If $\pi_\theta$ belongs to Case (1), then there is a unique band $\beta \in \A_\theta \setminus \A_0$ that satisfies Inequality~\eqref{e1}. By assumption, the largest column is not in $\A_0$; hence the column $Q_\theta(\beta)$ has to be the largest. At every stage before $\pi_\theta$, no column in $\A \setminus \A_0$ satisfies Inequality~\eqref{e1} and no column in $\A_0$ moves a column in $\A \setminus \A_0$. So we get the bound
\[
\vert Q_\theta(\beta)\vert <  2 \max_{\alpha \in \A_0} \vert Q_\iota(\alpha) \vert < (1+M) \max_{\alpha \in \A_0} \vert Q_\iota(\alpha) \vert
\]
Combining with Inequality~\eqref{upper} results in the upper bound
\[
\vert Q_\theta(\beta) \vert < (1+M)\max_{\alpha \in \A_0} \vert Q_\iota(\alpha) \vert < C_1'(1+M) \max_{\alpha \in \A_0} \vert Q(\alpha) \vert
\]
and so the collection $\A_\theta$ is $C_0C_1'(1+M)$-distributed.

\subsection*{\it Case 2:} If $\pi_\theta$ belongs to Case (2), then in the final split before $\pi_\theta$, some band $\alpha$ in $\A_0$ is split by some band $\beta$ in $\A \setminus \A_0$. So it satisfies~\eqref{e2}. Along the sequence $\pi_\iota \to \cdots \to \pi_\theta$, no band in $\A \setminus \A_0$ satisfies Inequality~\eqref{e1} at any intermediate stage. This implies that the column $Q_\theta(\beta)$ has to be the largest in norm. It also implies that
\begin{equation}\label{e6}
\vert Q_{\theta-1}(\beta) \vert < \max_{\alpha \in \A_0} \vert Q_\iota(\alpha) \vert
\end{equation}
Inequalities~\eqref{upper} and~\eqref{e6} give the upper bound
\begin{equation}\label{e7}
\vert Q_\theta(\beta) \vert < (1+M) \max_{\alpha \in \A_0} \vert Q_\iota(\alpha) \vert < C_1'(1+M) \max_{\alpha \in \A_0} \vert Q(\alpha) \vert
\end{equation}
and so the collection $\A_\theta$ is $C_0C_1'(1+M)$-distributed.
\end{proof}
Setting $c_1 = (1/4) c_1'$ and $C_1 = C_0C_1'(1+ M)$, we conclude the proof of the Proposition~\ref{inductive-step}.

\end{proof}

\noindent Iterating Proposition~\ref{inductive-step}, we get

\begin{proposition}\label{equidis}
There exist constants $c \in (0,1)$ and $C > 1$ that depend only on $d$, and a finite number of future stages $\pi_\tau$ after $\pi$, such that
\begin{enumerate}
\item The sets $\Gamma Q_\tau(W_\tau)$
have disjoint interiors and
\[
\sum_{\tau} \ell(\Gamma Q_\tau(W_\tau)) \geq c \cdot \ell(\Gamma Q(W))
\]
\item Each stage $\pi_\tau$ is $C$-distributed
\end{enumerate}
\end{proposition}

\noindent To prove Theorem~\ref{Uniform-distortion}, apply Proposition~\ref{equidis} to $\pi$ to get a set of stages, which we now denote by $\pi^{(1)}_\tau$, that are $C$-distributed. Let $X^{(1)}$ be the union of the sets $\Gamma Q_\tau(W_\tau)$. Consider the complement $\Gamma Q(W) \setminus X^{(1)}$. Write it as a union of stages with disjoint interiors. Apply Proposition~\ref{equidis} to each of them to get a new set of stages, which we denote by $\pi^{(2)}_\tau$, that are $C$-distributed. Let $X^{(2)}$ denote the union of the sets that correspond to the new $C$-distributed stages. Now consider the complement $\Gamma Q(W)\setminus X^{(1)} \cup X^{(2)}$ and iterate the process.

As a result, we construct an infinite sequence of sets $X^{(n)}$ with pairwise intersections of measure zero, such that they are a union of $C$-distributed stages. The set of $\x \in \Gamma Q(W)$ whose expansion gets $C$-distributed at some stage after $\pi$, is equal to the infinite union of the sets $X^{(n)}$. Let
\[
Y^{(n)} = \Gamma Q(W) \setminus \bigcup_{i=1}^{n-1} X^{(i)}
\]
By Proposition~\ref{equidis}, we know that $\ell(Y^{(n)}) \leq (1-c) \ell(Y^{(n-1)})$. Hence by induction
\[
\lim_{n \to \infty} \ell(Y^{(n)}) = 0
\]
So the infinite union of the sets $X^{(n)}$ has full measure which concludes the proof of Theorem~\ref{Uniform-distortion}. An immediate consequence of Theorem~\ref{Uniform-distortion} is

\begin{theorem}\label{uniform-infinite}
For almost every $\x \in W_0$, the expansion of $\x$ becomes $C$-distributed infinitely often.
\end{theorem}
\begin{proof}
Theorem~\ref{Uniform-distortion} says that the set whose expansion never gets $C$-distributed is measure zero. The set of $\x$ whose expansion gets $C$-distributed finitely many times is a countable union of such measure zero sets. Hence it is measure zero, proving the theorem.
\end{proof}

\section{Normality}\label{proof-of-normality}
In this section, we prove normality, i.e.~Theorem~\ref{normality}. For every $\pi'$ in the attractor, choose a {\em shortest} splitting sequence from $\pi'$ to $\pi_0$. This gives a finite number of chosen splitting sequences. This means that for any $C$-distributed $\pi'$
\begin{enumerate}
\item With probability bounded below by a constant $c'>0$ that depends only on $d$, it is possible  to split from the $C$-distributed $\pi'$ to a stage that ends at $\pi_0$ using the chosen sequence chosen above .
\item Splitting from $\pi'$ to a stage as above, introduces a bounded amount of distortion in the measure, where the bound depends only on $d$.
\end{enumerate}
The statements above imply that in Proposition~\ref{equidis}, for different constants $c \in (0,1)$ and $C>1$ which depend only on $d$, one gets $C$-distributed with a relative probability greater than $c$, such that the $C$-distributed stages have generalized permutation $\pi_0$. In other words, the uniformly distorted stages in Theorem~\ref{Uniform-distortion} can be assumed to have generalized permutation $\pi_0$. Moreover, because of this additional conclusion in Theorem~\ref{Uniform-distortion} and the fact that we are in the attractor, it is enough to prove normality for finite sequences that start from $\pi_0$.

Let $\jmath = \pi_0 \to \cdots \to \pi$ be a finite sequence of splits. We prove the stronger form of normality stated below, which is necessary to prove unique ergodicity.

\begin{theorem}[Strong Normality]\label{Strong Normality}
In almost every expansion, for any finite sequence $\jmath$ starting from $\pi_0$, there are infinitely many instances in which $\jmath$ immediately follows a $C$-distributed stage.
\end{theorem}
\begin{proof}
By Theorem~\ref{uniform-infinite}, for every $k \in \N$, there is a $k$-th instance of $C$-distribution in the expansion of almost every $\x \in W_0$. Let $\pi_{\x,n_k}$ be the $C$-distributed stage that is the $k$-th instance. For a finite sequence $\jmath$ in the attractor, let $X_k$ denote the set of all $\x$ for which the expansion after $\pi_{\x,n_k}$ begins with $\jmath$. Recall from Section~\ref{Jacobian}, that a $C$-distributed stage is $C^{d-1}$-uniformly distorted. So the probability that $\jmath$ follows any $C$-distributed stage is, up to a constant that depends only on $C$ and $d$, the same as the probability that an expansion begins with $\jmath$. Let this probability be $\mu>0$. So we have $\ell(X_k) \approx \mu$. Since
\[
\sum_{k=1}^{\infty} \ell(X_k) = \infty
\]
by the standard Borel-Cantelli Lemma, the set of $\x$ that belong to infinitely many $X_k$ has full measure.
\end{proof}

\section{Unique Ergodicity}\label{unique-ergodicity}

By Proposition~\ref{inv-measures}, to show unique ergodicity it is enough to show
\begin{theorem}\label{ergodicity}
For almost every $\x$ in $W_0$, the nested sequence of sets $\Gamma Q_n(W_n)$ given by the expansion of $\x$ satisfy
\[
\bigcap \Gamma Q_n(W_n) = \x
\]
In other words, almost surely the infinite intersection nests down to a point.
\end{theorem}
\begin{proof}
It is enough to prove that almost surely, the sets $\Gamma Q_n(\Delta)$ nest down to a point. To do this, suppose there exists a finite sequence $\jmath: \pi_0 \to \cdots \to \pi_j$ with the {\it diameter shrinking} property: There exists a universal positive constant $R<1$ such that when any $C$-distributed stage $\pi$ is followed by $\jmath$, the diameter becomes less than the diameter at $\pi$ by a factor less than $R$. By strong normality, in almost every expansion, the sequence $\jmath$ follows a $C$-distributed stage infinitely often. This implies that for almost every $\x$, the diameter shrinking occurs infinitely often and so the infinite intersection of the sets $\Gamma Q_n(\Delta)$ has diameter zero, i.e.~it is a point.

It is not immediate from the definition of $C$-distribution that a sequence $\jmath$ that shrinks diameter of the initial stage by $R<1$ also shrinks diameter at every $C$-distributed stage by $R$ (up to a constant that depends only on $C$ and $d$), i.e.~that $\jmath$ has the diameter shrinking property. The subtle issue is that a $C$-distributed stage is defined using the Jacobian. So, it is not immediate that the distortion is uniform along line segments in $\Delta$. The following lemma verifies that for a $C$-distributed stage, the distortion along line segments in $\Delta$ is $C^2$-uniform.

Let $\pi$ be a $C$-distributed stage and let $Q$ be the associated matrix. Let $L$ be a line segment in $\Delta$ with endpoints $\mathbf{u}_1$ and $\mathbf{u}_2$ in $\partial \Delta$. Denote the unit tangent vector along $L$ from $\mathbf{u}_1$ to $\mathbf{u}_2$ by $v$. Let $\w_1 = Q \mathbf{u}_1$ and $\w_2 = Q \mathbf{u}_2$. As in Section~\ref{Jacobian}, denote the full derivative of $\Gamma Q$ as a map from $\Delta$ to itself by $\dr \Gamma Q$.
\begin{lemma}\label{line-distortion}
For any pair of points $\y_1$ and $\y_2$ in $L$,
\[
\frac{1}{C^2} \vert \dr \Gamma Q(\y_1)(v) \vert < \vert \dr \Gamma Q(\y_2)(v) \vert < C^2 \vert \dr \Gamma Q(\y_1)(v) \vert
\]
\end{lemma}
\begin{proof}
Parameterize $Q(L)$ by $f(t) = (1-t) \w_1 + t\w_2$. Then $g(t) = f(t)/ \vert f(t) \vert$ is a parametrization of $\Gamma Q(L)$. By additivity of the norm, $\vert f(t) \vert = (1-t) \vert \w_1 \vert + t \vert \w_2 \vert$. Computing the derivative with respect to $t$, we get $\vert f(t) \vert' = \vert \w_2 \vert - \vert \w_1 \vert$
This gives
\begin{eqnarray*}
g'(t) &=& \frac{\vert f(t) \vert f'(t) - f(t) \vert f(t) \vert'}{\vert f(t) \vert^2}\\
&=& \frac{\left((1-t)\vert \w_1 \vert + t \vert \w_2 \vert \right)\left(\w_2 - \w_1 \right)- \left((1-t)\w_1 + t\w_2\right)\left( \vert \w_2 \vert - \vert \w_1 \vert \right)}{\vert f(t) \vert^2}\\
&=& \frac{\vert \w_1 \vert \w_2 - \vert \w_2 \vert \w_1}{\vert f(t) \vert^2}
\end{eqnarray*}
This means that for any pair of values $t_1$ and $t_2$ in the interval $[0,1]$,
\begin{equation}\label{line-est}
\frac{1}{C^2} < \frac{\vert g'(t_1) \vert}{\vert g'(t_2) \vert} < C^2
\end{equation}
Let $h(t)= (1-t) \mathbf{u}_1 + t \mathbf{u}_2$ parameterize $L$. The map $\Gamma Q$ restricted to $L$ is given by $\Gamma Q(\y) = g \circ h^{-1}(\y)$. Since $h$ is a parametrization of $L$ with constant speed, the estimate~\eqref{line-est} concludes the proof.
\end{proof}

Similar to Corollary 4 from Section 4.3 of Yoccoz \cite{Yoc}, there exists a splitting sequence $\jmath$ from $\pi_0$ whose associated matrix has every entry positive. For non-classical interval exchanges, the corollary is a consequence of Proposition~\ref{critical} using an argument identical to the original proof in Yoccoz.  The corollary implies that the image of $\Delta$ under the associated map is in the interior of $\Delta$. We fix this sequence $\jmath$ for the rest of this section.

Let $t_\alpha$ be the vertices $\Gamma Q_\jmath(\alpha)$ of $\Gamma Q_\jmath(\Delta)$. Let $s$ denote the minimum distance between $t_\alpha$ and $\partial \Delta$. The number $s$ is related to the largest entry of $Q_\jmath$.

Extend the edge $[t_\alpha,t_\beta]$ of $\Gamma Q_\jmath(\alpha)$ in either direction to a line segment $L[\alpha , \beta]$ with corresponding endpoints $u_\alpha$ and $u_\beta$ in $\partial \Delta$. Then the distances $d(t_\alpha,u_\alpha)$ and $d(t_\beta, u_\beta)$ are greater than $s$. So the length of $[t_\alpha ,t_\beta]$ is at most $(1-2s)$ times the length of $L[\alpha,\beta]$.

\begin{proposition}\label{shrinkdiameter}
Suppose $\pi$ is a $C$-distributed stage and let $\pi'$ be the stage obtained by $\jmath$ following $\pi$. Then there exists a positive constant $R<1$ that depends on $s, C$ and $d$ such that
\[
\text{diam} \left(\Gamma Q'(\Delta) \right) < R \hskip 3pt \text{diam} \left( \Gamma Q(\Delta) \right)
\]
\end{proposition}

\begin{proof}
Since $\Gamma Q'(\Delta)$ is a simplex, the longest line segment in it is a side joining a pair of vertices. Since $Q' = Q \cdot Q_j$, we assume that this side is $\Gamma Q([t_\alpha,t_\beta])$. Let $v$ be the unit tangent vector along the segment $L[\alpha ,\beta]$. By Lemma~\ref{line-distortion}, the distortion of $v$ is $C^2$-uniform on $L[\alpha, \beta]$. Combining this with the fact that the length of $[t_\alpha, t_\beta]$ is at most $(1-2s)$ times the length of $L[\alpha,\beta]$, we get the distance estimate
\[
\frac{d\left(\Gamma Q(t_\alpha),\Gamma Q(t_\beta)\right)}{d\left(\Gamma Q(u_\alpha), \Gamma Q(u_\beta)\right)} < \frac{C^2(1-2s)}{C^2(1-2s)+2s}
\]
Taking $R$ to be the right hand side in the above inequality, we finish the proof of Proposition~\ref{shrinkdiameter}.
\end{proof}

Strong normality applied to the diameter shrinking sequence $\jmath$ implies that diameter shrinking happens infinitely often in almost every expansion, and so almost surely, the infinite intersection is a point.
\end{proof}
Finally, Theorem~\ref{ergodicity} implies that for almost every $\x$, there is a unique transverse probability measure. Consequently, almost every non-classical interval exchange is uniquely ergodic.

\section{Rauzy map on the parameter space}\label{Rauzy-par}

Consider the full parameter space for a strongly irreducible non-classical interval exchange. It is a disjoint union of configuration spaces, each defined by a generalized permutation in the attractor. Rauzy induction induces a map from this parameter space to itself described in Section~\ref{dynamics}. Each configuration space in the full parameter space carries a natural Lebesgue measure. A consequence of Theorems~\ref{uniform-infinite} and~\ref{ergodicity} is

\begin{theorem}\label{parameter-ergodic}
The Lebesgue measure on the full parameter space is ergodic with respect to the Rauzy map.
\end{theorem}
\begin{proof}
Let $A$ be a Borel set in the full parameter space invariant under the Rauzy map and suppose that $A$ is not full measure. Then $A$ is not full measure in some configuration space $W$ defined by some generalized permutation $\pi$. Let $\x$ be a point of Lebesgue density for the complement $W \setminus A$. Given $1/2 > \epsilon > 0$, there is a ball $B(\x)$ centered at $\x$ such that the proportion of $A$ in this ball is less than $\epsilon$. 

Starting with $\pi$ as the initial generalized permutation, Theorems~\ref{uniform-infinite} and~\ref{ergodicity} imply that the ball $B(\x)$ can be closely approximated by a union of $C$-distributed stages contained in $B(\x)$. Furthermore, the argument at the beginning of Section~\ref{proof-of-normality} implies that given a generalized permutation $\pi'$ in the attractor, the $C$-distributed stages approximating $B(\x)$ can be all arranged to have generalized permutation $\pi'$. Let us suppose the approximation is chosen so that the union of these $C$-distributed stages has measure greater than $(1-\epsilon)\ell(B(\x))$. If the proportion of $A$ in each of these $C$-distributed stages were greater than $2 \epsilon$, then that would imply that the proportion of $A$ in $B(\x)$ is greater than $2\epsilon (1-\epsilon)> \epsilon$, a contradiction. This means that there exists at least one $C$-distributed stage such that the proportion of $A$ in it is at most $2 \epsilon$. 

Let $Q$ be the matrix associated to this particular stage and let $W'$ be the configuration space defined by $\pi'$. By the invariance of $A$ under the Rauzy map
\[
\frac{\ell(A \cap W')}{\ell(W')}= \frac{\ell(A \cap \Gamma Q(W'))}{\ell(\Gamma Q(W'))} \leqslant 2 \epsilon
\]
i.e., the proportion of $A$ in $W'$ is at most $2 \epsilon$. Since $\epsilon$ can be made arbitrarily small and $\pi'$ was chosen to be any generalized permutation in the attractor, the conclusion is that $A$ must have measure zero. Thus the invariant sets for the Rauzy map on the parameter space have zero or full measure i.e. the Rauzy map is ergodic.
\end{proof}

\end{document}